\documentclass[12pt]{article}
\usepackage{amssymb,amsmath,epsfig,amscd,eucal,psfrag,amsthm,enumerate,tikz-cd,MnSymbol}
\usepackage[T1]{fontenc}
\usepackage[latin9]{inputenc}
\usepackage{graphicx,color}

\newif\ifshowcomments

\showcommentstrue

\newtheorem{theorem}{Theorem}[section]

\newtheorem{lemma}[theorem]{Lemma}
\newtheorem{proposition}[theorem]{Proposition}

\newtheorem{corollary}[theorem]{Corollary}
\newtheorem{conjecture}[theorem]{Conjecture}

\newtheorem{problem}[theorem]{Problem}
\newtheorem{letterthm}{Theorem}

\newtheorem{letterprop}[letterthm]{Proposition}

\makeatletter
\newtheorem*{rep@theorem}{\rep@title}
\newcommand{\newreptheorem}[2]{%
\newenvironment{rep#1}[1]{%
 \def\rep@title{#2 \ref{##1}}%
 \begin{rep@theorem}}%
 {\end{rep@theorem}}}
\makeatother

\newreptheorem{theorem}{Theorem}
\newreptheorem{cor}{Corollary}

\theoremstyle{definition}
\newtheorem{definition}[theorem]{Definition}

\theoremstyle{remark}
\newtheorem{remark}[theorem]{Remark}
\newtheorem{example}[theorem]{Example}

\newcommand{\N}{\mathbb{N}}

\newcommand{\Q}{\mathbb{Q}}
\newcommand{\R}{\mathbb{R}}

\newcommand{\Z}{\mathbb{Z}}

\newcommand{\Lk}{\mathrm{Lk}}

\newcommand{\cl}{\mathrm{cl}}
\newcommand{\scl}{\mathrm{scl}}

\newcommand{\rk}{\mathrm{rk}}

\newcommand{\Area}{\mathrm{Area}}
\newcommand{\Irred}{\mathrm{Irred}}

\newcommand{\Surf}{\mathrm{Surf}}

\newcommand{\curlyA}{\mathcal{A}}
\newcommand{\curlyB}{\mathcal{B}}

\newcommand{\curlyG}{\mathcal{G}}
\newcommand{\curlyH}{\mathcal{H}}

\newcommand{\curlyP}{\mathcal{P}}

\newcommand{\curlyX}{\mathcal{X}}

\newcommand{\opp}[1]{{#1}^*}

\ifshowcomments
\newcommand{\comment}[1]{\marginpar{\sffamily{\tiny #1
\par}\normalfont}}
\else
\newcommand{\comment}[1]{}
\fi

\usepackage{hyperref}

\input xy
\xyoption{all}

\title{Rational curvature invariants for 2-complexes}
\author{Henry Wilton}

\newcommand{\Addresses}{{
  \bigskip
  \footnotesize

  \textsc{DPMMS, Centre for Mathematical Sciences, Wilberforce Road, Cambridge, CB3 0WB, UK}\par\nopagebreak
  \textit{E-mail address:} \texttt{h.wilton@maths.cam.ac.uk}

}}

\begin{document}

\maketitle

\begin{abstract}
New invariants for 2-dimensional cell complexes are defined, which can be interpreted as curvature bounds. These invariants are proved to be rational and computable in a companion article.  This document is a survey that collects theorems about these invariants, computes examples, and lays out a programme of conjectures. 
\end{abstract}

\tableofcontents

\section{Introduction}

\subsection{Historical background}

In a 1911 paper, Max Dehn initiated the study of infinite discrete groups of symmetry: he posed the \emph{word problem}, and solved it in the case of the fundamental group of a closed surface. Already, in Dehn's seminal work, there is a tension between the combinatorial and geometric approaches: he gave a geometric solution to the word problem for surfaces in 1911 \cite{dehn_unendliche_1911}, then followed it with a combinatorial solution the following year \cite{dehn_transformation_1912}. This tension persists today.

The combinatorial approach rose to the fore in the immediate aftermath of Dehn's work, and the field of \emph{combinatorial group theory} was born. Combinatorial group theory can be characterised as the study of groups defined by presentations, i.e.\ sets of generators subject to relations. The paradigmatic example is the theory of groups with one relator, developed by Dehn's student Wilhelm Magnus in the 1930s. Magnus solved the word problem in all one-relator groups using combinatorial methods \cite{magnus_identitatsproblem_1932}, and combinatorial group theory remained the dominant method of studying infinite discrete groups for half a century.

Dehn himself seems to have had reservations about the combinatorial method: he famously said `So you went about it blind!' when he saw Magnus' combinatorial proof \cite[p.126]{nyberg-brodda_newman_2021}. Nevertheless, the geometric method took a back seat until the 1980s, when Mikhail Gromov, a differential geometer, repopularised Dehn's work. Gromov observed that Dehn's 1911 solution to the word problem, using the geometry of the hyperbolic plane, applies to a much larger class of group: the \emph{(word-)hyperbolic groups}, which enjoy a coarse form of negative curvature \cite{gromov_hyperbolic_1987}. So began the field of \emph{geometric group theory}, which studies groups via their actions on metric spaces. Hyperbolic groups became the paradigmatic examples, and geometric group theory is now the dominant method of studying infinite discrete groups.

While many examples fit into both the combinatorial and geometric paradigms, no over-arching theory unifies the two. The geometric point of view is satisfying and offers very general arguments for certain classes of groups, but leaves other groups untouched. Baumslag--Solitar groups, as well as other small and natural examples, which are well handled by the combinatorial approach, are pathological from the geometric perspective. 

Our purpose here is to propose a new approach that enjoys some of the advantages of both. Like Dehn, we are motivated by topology, and therefore we will work with 2-dimensional cell complexes rather than presentations. (There is no loss of generality, since a presentation naturally defines a presentation complex.) To each such complex $X$ we will associate several numerical invariants, which synthesise several recent strands of research in group theory and low-dimensional topology. They encompass both one-relator groups and many natural geometric conditions, including small-cancellation conditions and some of Gromov's metric curvature conditions. Despite being defined combinatorially, we argue that these invariants should be thought of geometrically, as curvature bounds.

\subsection{Curvature invariants}

We start with the following crude definition. The \emph{average curvature} of a finite 2-complex $X$ is defined to be
\[
\kappa(X):=\frac{\chi(X)}{\Area(X)}
\]
where $\chi(X)$ is the Euler characteristic of $X$ and $\Area(X)$ is  the number of 2-cells of $X$. This definition is motivated by the Gauss--Bonnet theorem, which tells us that the average Gaussian curvature of a closed Riemannian surface $S$ is $2\pi\chi(S)/\Area(S)$. 

The average curvature is not very different from the Euler characteristic, and so will not on its own contain much information about $X$ or its fundamental group. 
However, it is a familiar idea in topology and geometry that we may define interesting invariants of a space $X$ by probing it with maps from other spaces $Y$. Following this idea, we define more refined invariants by extremising $\kappa$ over certain classes of 2-complexes $Y$ that map to $X$. (In fact, it turns out to be fruitful to allow $Y$ to be a slightly more general object called a \emph{branched} 2-complex.\footnote{See \S\ref{sec: Branching and area} for the definition of branched 2-complexes and \S\ref{sec: Curvature invariants} for the definition of $\kappa(X)$ in that context.})

There are various ways in which it may be possible to simplify a 2-complex $X$, such as if $X$ has a free face, a separating vertex etc. A 2-complex that cannot be simplified in any of these ways is called \emph{irreducible}. Our first curvature invariant is the \emph{maximal irreducible curvature}
\[
\rho_+(X):=\sup_{Y\in\Irred(X)}\kappa(Y)
\]
where $\Irred(X)$ consists of all finite, irreducible, branched 2-complexes $Y$ admitting essential maps $Y\to X$. Likewise, the \emph{minimal irreducible curvature} $\rho_-(X)$ is the corresponding infimum.\footnote{See \S\ref{sec: Branching and area} for the definition of essential maps, \S\ref{sec: Irreducible complexes} for the definition of irreducible complexes, and \S\ref{sec: Curvature invariants} for the precise definition of $\rho_{\pm}(X)$.}

Although we believe that these definitions have not appeared in the literature before, they are prefigured in various places. In the early 2000s, Wise defined the \emph{non-positive-} and \emph{negative-immersions} properties for 2-complexes, as a route towards studying the fundamental group $\pi_1(X)$, especially in the case where $X$ is a one-relator presentation complex \cite{wise_nonpositive_2003,wise_sectional_2004}. These properties are in a similar spirit to the definitions of non-positive and negative irreducible curvature, respectively, and indeed this document can be thought of as a modification of the programme suggested by Wise in \cite{wise_invitation_2020} and \cite{wise_coherence_2022}. The independent proofs that one-relator groups have non-positive immersions by Helfer and Wise \cite{helfer_counting_2016} and by Louder and the author \cite{louder_stackings_2017}, and the subsequent work on various notions of negative immersions for one-relator groups \cite{louder_one-relator_2020,louder_uniform_2024,louder_negative_2022,wise_coherence_2022}, provide large classes of examples to which these definitions apply. Puder's \emph{primitivity rank} \cite{puder_primitive_2014} is another closely related invariant, discussed in \S\ref{sec: Primitivity rank}.

Irreducible 2-complexes are the largest naturally defined class of finite 2-complexes over which we can usefully extremise $\kappa$. We can define a second pair of invariants by extremising $\kappa$ over the smallest naturally defined class of irreducible 2-complexes mapping to $X$, namely closed surfaces. The \emph{maximal surface curvature} of $X$ is defined to be
\[
\sigma_+(X):=\sup_{Y\in\Surf(X)}\kappa(Y)
\]
where $\Surf(X)$ consists of all closed surfaces $Y$ admitting essential maps $Y\to X$, and the \emph{minimal surface curvature}, $\sigma_-(X)$, is defined to be the corresponding infimum.\footnote{See \S\ref{sec: Curvature invariants} for the definitions  $\sigma_{\pm}(X)$.}  In contrast to the case of irreducible curvature, it is not immediately obvious that $\Surf(X)$ is non-empty for most complexes $X$. However, it follows from the results of \cite{wilton_essential_2018} that  $\Surf(X)\neq\varnothing$ whenever $X$ is irreducible, so the invariants $\sigma_\pm(X)$ are well defined in that case. (See Theorem \ref{thm: Essential surfaces exist}.)

Again, although we believe these definitions to be new, they do have various precedents in the literature.  Wise's \emph{planar sectional curvature} \cite{wise_sectional_2004} provides sufficient conditions for $X$ to have negative or non-positive surface curvature. In a different direction, the quantity $(1-\sigma_+(X))/2$ can be regarded as an \emph{unoriented} version of stable commutator length; see \S\ref{sec: Stable commutator length} for a discussion.  The two most important theorems about stable commutator length in free groups -- the Duncan--Howie lower bound $\scl(w)\geq1/2$ \cite{duncan_genus_1991} and Calegari's rationality theorem \cite{calegari_stable_2009} -- prefigure some of the most important theorems that we discuss here.

The fundamental result about our invariants is the \emph{rationality theorem}, which tells us that the extrema in the above definitions are realised, and in so doing justifies the words `maximal' and `minimal' in the names of the curvatures. The proof realises $\rho_\pm(X)$ and $\sigma_\pm(X)$ as extrema of rational linear programming problems, along the lines of Calegari's rationality theorem for stable commutator length in free groups \cite{calegari_stable_2009}. In particular, as well as being rational, these invariants can be computed, at least in principle.

\begin{theorem}[Rationality theorem]\label{thm: Rationality theorem}
If $X$ is a finite, irreducible 2-complex, then
\[
\rho_+(X)=\max_{Y\in \Irred(X)}\kappa(Y)
\]
and
\[
\sigma_+(X)=\max_{Y\in \Surf(X)}\kappa(Y)\,.
\]
Likewise, $\rho_-(X)$ and $\sigma_-(X)$ are the corresponding minima. In particular, $\rho_\pm(X),\sigma_\pm(X)\in\Q$. Furthermore, they can be computed algorithmically from $X$. 
\end{theorem}

The proof of Theorem \ref{thm: Rationality theorem} is long and technical, and is therefore in a self-contained paper \cite{wilton_rationality_2022}. 

Theorem \ref{thm: Rationality theorem} provokes the question of which ranges of values the curvature invariants can take.  The next theorem constrains them inside a certain three-dimensional simplex.

\begin{letterthm}[Curvature inequalities]\label{thm: Curvature inequalities}
If $X$ is a finite, irreducible 2-complex, then
\[
-\infty<\rho_-(X)=\sigma_-(X)\leq\sigma_+(X)\leq\rho_+(X)\leq 1\,.
\]
\end{letterthm}

Of the various relationships contained in Theorem \ref{thm: Curvature inequalities}, much the most striking and difficult to prove is that $\rho_-(X)=\sigma_-(X)$, which means that the minimal irreducible curvature is always realised by an essential map from a branched surface.

These results motivate a broad programme, which can be summarised in two problems that we call, for ease of reference, the \emph{botany problem} and the \emph{geography problem}. Given the historical emphasis in geometric group theory on upper curvature bounds, the cases of \emph{non-positive curvature} (i.e.\ $\rho_+(X)\leq 0$ or $\sigma_+(X)\leq 0$) and \emph{negative curvature} (i.e.\ $\rho_+(X)<0$ or $\sigma_+(X)< 0$) are of particular interest.

The botany problem asks for examples of 2-complexes satisfying curvature bounds.

\begin{problem}[Botany problem]\label{prob: Botany problem}
For natural classes of 2-complexes $X$, compute the invariants $\rho_+(X)$, $\sigma_+(X)$. Find examples of 2-complexes $X$ that are \emph{non-positively curved} or \emph{negatively curved}.
\end{problem}

In this paper, we collect several solutions to the botany problem. Some of the most significant are summarised in the next theorem.

\begin{letterthm}[Examples]\label{thm: Botany results}
Let $X$ be a finite, irreducible 2-complex.
\begin{enumerate}[(i)]
\item If $X$ is locally CAT(0) or satisfies the C(6) small-cancellation condition then $\sigma_+(X)\leq 0$. If $X$ is locally CAT(-1) or C(7) then $\sigma_+(X)< 0$.
\item If $X$ is the presentation complex of a torsion-free one-relator group then $\rho_+(X)\leq 0$.
\item If $X$ is is the spine of a compact 3-manifold\footnote{Formally, this hypothesis asks that $X$ should be \emph{geometric} in the sense of Definition \ref{def: 3-manifold spine}.} then $\sigma_+(X)=\rho_+(X)$.
\end{enumerate}
\end{letterthm}

Item (i) of the theorem is the content of Corollaries \ref{cor: Surface curvature of CAT(0) 2-complexes} and \ref{cor: Surface curvature of CAT(-1) 2-complexes} together with Propositions \ref{prop: Surface curvature of C(6) complexes} and \ref{prop: Surface curvature of C(7) complexes}, item (ii) is Theorem \ref{thm: Irreducible curvature of one-relator groups}, and item (iii) is Corollary \ref{cor: Curvatures of 3-manifold spines}.

These results should be compared with the standard notions of curvature in geometric group theory.  Item (i) of Theorem \ref{thm: Botany results} suggests that $\sigma_+(X)$ carries similar information to the CAT($\kappa$) and small-cancellation notions of curvature, which are already widely used throughout geometric group theory. On the other hand, item (ii) implies that our notions of non-positive curvature encompass some examples, such as the presentation complexes of Baumslag--Solitar groups, that do not satisfy any of the standard notions of non-positive curvature in geometric group theory. Finally, item (iii) suggests that irreducible and surface curvatures capture information that explains some of the special properties of 3-manifold groups.

These examples are discussed in \S\ref{sec: Examples}, along with some conjectures which posit that the scope of our invariants can be expanded to capture the behaviour of many of the other favourite low-dimensional examples of geometric group theory.   Figure \ref{fig: Conjectural botany} illustrates examples of 2-complexes $X$ with various combinations of $\rho_+(X)$ and $\sigma_+(X)$.

Our second broad problem is the geography problem, which asks for constraints on 2-complexes that satisfy curvature bounds.

\begin{problem}[Geography problem]\label{prob: Geography problem}
Prove theorems about 2-complexes $X$ satisfying constraints on $\rho_\pm(X)$ or $\sigma_\pm(X)$, especially under the hypotheses that either $\rho_+(X)$ or $\sigma_+(X)$ is negative or non-positive.
\end{problem}

We offer a few contributions to the geography problem here.  The first is an easy observation, which characterises \emph{concise} 2-complexes.

\begin{definition}\label{defn: Concise 2-complex}
A 2-complex is called \emph{concise} unless it has either two 2-cells with homotopic attaching maps or a single 2-cell whose boundary is a proper power. 
\end{definition}

Equivalently, for a group presentation $\langle a_1,\ldots,a_m\mid r_1,\ldots,r_n\rangle$, the associated presentation complex $X$ is concise if the family of cyclic subgroups $\{\langle r_1\rangle,\ldots,\langle r_n\rangle\}$ of the free group $F_m=\langle a_1,\ldots,a_m\rangle$ is malnormal.  This definition is slightly stronger than the definition of a concise 2-complex used by Chiswell--Collins--Huebschmann \cite{chiswell_aspherical_1981}, but is convenient for our purposes.

\begin{letterprop}\label{prop: rho_+=1}
For $X$ a finite, irreducible 2-complex, the following are equivalent:
\begin{enumerate}[(i)]
\item \label{prop: rho_+=1 item (i)} $\rho_+(X)=1$;
\item \label{prop: rho_+=1 item (ii)} $\sigma_+(X)=1$;
\item \label{prop: rho_+=1 item (iii)}$X$ is not concise.
\end{enumerate}
\end{letterprop}

 The proof is given in \S\ref{sec: Concise 2-complexes}. The proposition implies that there is no finite 2-complex $X$ with $\sigma_+(X)<\rho_+(X)=1$.  I expect this to be the only collection of curvature constraints compatible with Theorem \ref{thm: Curvature inequalities} that are not realised by any finite, irreducible 2-complex.
 
 The next contribution is more interesting, and illustrates the strength of these invariants. It is not difficult to see that any closed surface $X$ has constant irreducible curvature -- that is, $\rho_+(X)=\rho_-(X)$ (see Example \ref{eg: Surfaces}). Remarkably, 2-complexes with constant irreducible curvature can be completely classified, and they are all closely related to surfaces.  More precisely, they are \emph{irrigid}, which roughly means that they can be obtained from surfaces with boundary by gluing their boundary circles together.\footnote{See Definition \ref{def: Surface decomposition and irrigid complexes} for a precise definition.}

\begin{letterthm}\label{thm: Constant irreducible curvature}
Let $X$ be a finite, irreducible 2-complex. If $\rho_+(X)=\rho_-(X)$ then $X$ is irrigid.
\end{letterthm} 

It is doubtful that any other collection of curvature constraints leads to a complete characterisation beyond Proposition \ref{prop: rho_+=1} and Theorem \ref{thm: Constant irreducible curvature}. However, curvature constraints can have profound consequences for the group theory of the fundamental group. The next theorem collects some examples of such consequences, which all derive from the work of Wise, Louder and the author. Recall that a group is \emph{locally indicable} if every non-trivial finitely generated subgroup surjects $\mathbb{Z}$; it is \emph{2-free} if every 2-generator subgroup is free; and it is \emph{coherent} if every finitely generated subgroup is finitely presented.

\begin{letterthm}[Upper bounds on irreducible curvature]\label{thm: Consequences of upper bounds on irreducible curvature}
Let $X$ be a finite, irreducible 2-complex.
\begin{enumerate}[(i)]
\item If $\rho_+(X)\leq 0$ then $X$ is aspherical and $\pi_1(X)$ is locally indicable.
\item If $\rho_+(X)<0$ then $\pi_1(X)$ is 2-free and coherent.
\end{enumerate}
\end{letterthm}

Item (i) of Theorem \ref{thm: Consequences of upper bounds on irreducible curvature} is Theorem \ref{thm: Geography with non-positive irreducible curvature}, while item (ii) is a sample of the conclusions of Theorem \ref{thm: Subgroups of groups with negative irreducible curvature}. 

Conjecturally, upper bounds on irreducible curvature should have even stronger consequences; see, for instance, Conjectures \ref{conj: Non-positive irreducible curvature and coherence} and \ref{conj: Negative irreducible curvature implies hyperbolic and locally quasiconvex}. There are also corresponding conjectures about 2-complexes with upper bounds on their surface curvature; see Conjectures \ref{conj: Non-positive surface curvature implies aspherical}, \ref{conj: Non-positive surface curvature implies solvable word problem} and \ref{conj: Negative surface curvature implies hyperbolic}.  Figure \ref{fig: Conjectural geography} gives a conjectural picture of the properties of $X$ that are determined by $\rho_+(X)$ and $\sigma_+(X)$.

This paper is structured as follows. The basic definitions of branched 2-complexes and the maps between them are given in \S\ref{sec: Branching and area}. The fundamental notion of an irreducible complex is introduced in \S\ref{sec: Irreducible complexes}, and surfaces appear in \S\ref{sec: Surfaces}. The stage is then set to define $\rho_\pm(X)$ and $\sigma_\pm(X)$ in \S\ref{sec: Curvature invariants}.  Some immediate consequences of the definitions are explained in \S\ref{sec: Inequalities and first examples}. One advantage of these definitions over other similar curvature notions is that they are \emph{Nielsen invariants}, meaning that they do not change if the 1-skeleton is modified by a homotopy equivalence; this is explained in \S\ref{sec: Nielsen equivalence}. Theorems \ref{thm: Curvature inequalities} and \ref{thm: Constant irreducible curvature} are proved in \S\ref{sec: Irrigid complexes}, where irrigid complexes are discussed. Relationships with other similar invariants are discussed in \S\ref{sec: Relations to other invariants}. Finally, \S\ref{sec: Examples} addresses the botany problem by collecting examples, and \S\ref{sec: Properties of groups with curvature bounds} addresses the geography problem by collecting properties of 2-complexes with curvature bounds.

\subsection*{Acknowledgements}

This work would have been impossible without the input of Lars Louder. Our collaborative work on one-relator groups convinced me that these definitions were sufficiently broadly applicable to be worth studying in their own right. The debt that this work also owes to the vision of Dani Wise cannot be overstated -- it's clear from reading some of his early papers on coherence \cite{wise_nonpositive_2003,wise_sectional_2004} that he anticipated much of what is written here. Danny Calegari's ideas, and especially to his proof of the rationality theorem for stable commutator length in free groups \cite{calegari_stable_2009}, have also had a profound influence. 

Thanks are also due to Edgar Bering, Martin Blufstein, Koji Fujiwara, Daniel Groves, Nir Lazarovich, Jason Manning, Alexis Marchand, Doron Puder and Jacob Willis for useful conversations.

\section{Branching and area}\label{sec: Branching and area}

A 2-dimensional complex $X$ is the result of attaching 2-cells to a topological graph along their boundaries. This graph is the 1-skeleton $X_{(1)}$, and is constructed by attaching gluing edges from an edge set $E_X$ to a vertex set $V_X$.\footnote{Several formalisms for graphs can be used. For instance, $X_{(1)}$ can be taken to be a Serre graph, as in Stallings' definitive paper  \cite{stallings_topology_1983}.  Our discussion here does not rely on the details of the formalism chosen.}  The 2-cells of a 2-complex $X$ are called \emph{faces}, and the set of faces is denoted by $F_X$.  Formally, the faces can be modelled as circles mapping into $X_{(1)}$. After a homotopy, we may always assume that the attaching maps of the faces are immersions, at the cost of removing any face whose attaching maps is homotopically trivial in the 1-skeleton.

\begin{example}[Presentation complex]\label{eg: Presentation complex}
A group presentation $\curlyP=\langle A \mid R\rangle$ defines a presentation complex $X=X_\curlyP$. The 1-skeleton $X_{(1)}$ is a wedge of circles indexed by $A$. The faces correspond to the elements of $R$: each $r\in R$ defines a loop in $X_{(1)}$, and this loop defines the attaching map of the corresponding face.
\end{example}

The maps we allow  between 2-complexes are \emph{combinatorial}, meaning that the send vertices to vertices, edges to edges and faces to faces, all homeomorphically.

This much is standard, but it turns out to be fruitful to extend the category of 2-complexes that we work with in a slightly non-standard way. 

\begin{definition}\label{def: Branched 2-complex}
A \emph{branched 2-complex} $X$ is a 2-complex equipped with a notion of \emph{area}, namely a map
\[
\Area: F_X\to \Z_{> 0}\,.
\] 
The notation $\Area(X)$ is also used to denote the sum of the areas of the faces of $X$, and we insist that $\Area(X)> 0$.

We also allow a slightly wider selection of maps between branched 2-complexes. A \emph{branched morphism} between branched 2-complexes $\phi:Y\to X$ still sends vertices to vertices and edges to edges homeomorphically. It also send faces to faces, but on faces the maps are allowed a single ramification point in the centre; that is, on the faces, a map is modelled on the polynomial map
\[
z\mapsto z^m
\]
of the unit disc in the complex plane, for some positive integer $m$. In particular, $\phi$ assigns to each face $f$ of $Y$ a \emph{multiplicity} $m_\phi(f)\in \N$, namely the ramification index at the centre of $f$. However, the branching of the map $\phi$ is constrained to respect area in the following sense:
\[
\Area(f)=m_\phi(f)\Area(\phi(f))
\]
for each face $f\in F_Y$.
\end{definition}

Each 2-complex has a canonical structure as a branched 2-complex, obtained by setting the area of each face equal to 1. We will sometimes call such a 2-complex \emph{standard}, to emphasise that there is no branching. Note that a branched morphism between two standard 2-complexes is just a combinatorial map in the standard sense.

Branched coverings of surfaces provide a wealth of natural examples, as long as one is careful to place the branch points in the centres of faces.

\begin{example}[Branched self-coverings of the 2-sphere]\label{eg: Branched self-coverings of the 2-sphere}
For each $n\geq 1$, let $X_n$ be the cellulation of the 2-sphere with two faces glued along an equatorial circle with $n$ vertices and $n$ edges. Let $X=X_1$, and make it into a standard 2-complex by giving both faces area 1. There is an $n$-sheeted map $\phi_n:X_n\to X$ that wraps the equator of $X_n$ $n$ times around the equator of $X$. If we set the area of each face of $X_n$ equal to $n$, then the map $\phi_n$ is a branched morphism.
\end{example}

Our goal is to understand a 2-complex $X$ by analysing a large set of maps $Y\to X$, so it is of the utmost importance to constrain the kinds of maps we allow.

\begin{definition}[Branched coverings and immersions]\label{def: Branched coverings and immersions}
A branched map $\phi:Y\to X$ is a \emph{branched immersion} if it is locally injective away from the midpoints of the 2-cells; equivalently, a branched immersion is a map that is injective on the links of vertices. If $X$ and $Y$ are standard then $\phi$ is an \emph{immersion}, and is locally injective everywhere.

If $\phi$ is a bijection on the links of vertices, then $\phi$ is a \emph{branched covering}. An immersion which is a branched covering is exactly the usual topological notion of \emph{covering map}.
\end{definition}

Stallings famously explained how to study the subgroups of free groups by folding maps of graphs \cite[\S5.4]{stallings_topology_1983}. Stallings' folding procedure provides an algorithm to factor a map of finite graphs as a composition of a $\pi_1$-surjection and an immersion. Our set-up makes it easy to extend folding to 2-complexes: one just folds the 1-skeleta, and then identifies any 2-cells with the same image. The next lemma makes this precise. 

\begin{lemma}[Folding 2-complexes]\label{lem: Folding 2-complexes}
Every map of finite, connected, branched 2-complexes  $\phi:Y\to X$ factors as
\[
Y\stackrel{\phi_0}{\to}\bar{Y}\stackrel{\bar{\phi}}{\to} X
\]
where $\phi_0$ is a surjection on fundamental groups and $\bar{\phi}$ is a branched immersion.
\end{lemma}

Cousins of this lemma have already been used extensively \cite{louder_stackings_2017,louder_one-relator_2020,louder_negative_2022,louder_uniform_2024}. The reader is referred to \cite[Lemma 3.8]{wilton_rationality_2022} for a proof.

It would be convenient if we could always work with branched immersions. However, sometimes it is necessary to \emph{unfold} in order to understand the structure of a given complex, as the following simple example shows.

\begin{example}[Folded surface]\label{eg: Folded surfaces}
Consider a compact surface $S$ equipped with a simplicial triangulation (i.e.\ every edge and face has distinct vertices), and let $X$ be the result of folding a single pair of adjacent edges of $S$ at a common vertex. Then there are no branched immersions from closed surfaces to $X$.
\end{example}

In order to handle examples like this, we relax the definition of a branched immersion slightly to an \emph{essential map}. In related contexts, this definition already played important roles in \cite{wilton_essential_2018} and \cite{louder_uniform_2024}. The idea is that folding an essential map $Y\to X$ does not change any important information about $Y$.

\begin{definition}[Essential map]\label{def: Essential map}
A map of branched 2-complexes $\phi:Y\to X$ is \emph{essential} if, after folding, the resulting map $\phi_0:Y\to \bar{Y}$ satisfies the following conditions.
\begin{enumerate}[(i)]
\item The induced map on the 1-skeleton $\phi_0:Y_{(1)}\to \bar{Y}_{(1)}$ is a homotopy equivalence.
\item When restricted to the interiors of faces, $\phi_0$ is a homeomorphism. Equivalently,  $\phi_0$ preserves area.
\end{enumerate}
If $\phi$ is essential and, furthermore, $\bar{\phi}:\bar{Y}\to X$ is an isomorphism, then $\phi$ is called an \emph{essential equivalence}.
\end{definition}

Note that the fold map $S\to X$ from Example \ref{eg: Folded surfaces} is essential, indeed an essential equivalence. This illustrates an important principle: given a 2-complex $X$, there may be an essential equivalence $Y\to X$ such that the structure of $Y$ is more apparent than the structure of $X$.

\section{Group pairs}\label{sec: Group pairs}

It is often convenient to think about 2-complexes in a more algebraic way, provided by the following definition. Throughout this section, the 2-complex $X$ is assumed to be connected.

\begin{definition}[Group pair]\label{def: Group pair}
A \emph{group pair} $(F,\curlyA)$  consists of a finitely generated free group $F$ together with a cofinite action of $F$ on a set $\curlyA$ such that every stabiliser is infinite cyclic.
\end{definition}

There is a natural theory of group pairs, which it is not hard to develop on its own terms. However, the point of this definition is that it captures some of the underlying information of a 2-complex, and so we will develop the theory in parallel with the corresponding properties of 2-complexes and the maps between them.

\begin{definition}[Group pair of a 2-complex]
 Every finite, connected 2-complex $X$ has an \emph{associated group pair} $(F_X,\curlyA_X)$, defined as follows.  The free group $F_X$ is the fundamental group of the 1-skeleton $X_{(1)}$. The boundaries of the faces define a finite set $\{\langle w_1\rangle,\ldots,\langle w_n\rangle\}$ of infinite cyclic subgroups of $F_X$, each defined up to conjugacy. The set $\curlyA_X$ is then the disjoint union of the sets of left-cosets $F_X/\langle w_i\rangle$, with the natural action by left multiplication.
\end{definition}

\begin{remark}\label{rem: Concise group pairs}
Usually $X$ is assumed to be concise. This means that the set $\{\langle w_1\rangle,\ldots,\langle w_n\rangle\}$ is \emph{malnormal} -- that is, any intersection $g\langle w_i\rangle g^{-1}\cap \langle w_j\rangle$ is trivial unless $i=j$ and $g\in \langle w_i\rangle$. In this case, $\curlyA_X$ can be identified with the disjoint union of the sets of conjugates of the $\langle w_i\rangle$, with the natural action by conjugation. We give this slightly more complicated definition in order to handle 2-complexes that are not concise.
\end{remark}

A \emph{morphism} of a group pairs $\phi:(F,\curlyA)\to (G,\curlyB)$ consists of a homomorphism $\phi:F\to G$ and a set map $\phi:\curlyA\to\curlyB$ subject to the usual equivariance condition, i.e.\ that $\phi(g.a)=\phi(g)\phi(a)$.  Any map of 2-complexes $\phi:Y\to X$ gives rise to a natural map of the associated group pairs
\[
\phi_*: (F_Y,\curlyA_Y)\to (F_X,\curlyA_X)
\]
defined as follows. The group homomorphism $\phi_*:F_Y\to F_X$ is the map induced on fundamental groups of 1-skeleta by $\phi$, but there is also an equivariant map $\phi_*:\curlyA_Y\to\curlyA_X$: if $\phi_*$ sends the face corresponding to $\langle u_i\rangle$ to the face corresponding to $\langle w_j\rangle$ then, after choosing conjugacy representatives so that $\phi_*(u_i)\in \langle w_j\rangle$, the map $\curlyA_Y\to\curlyA_X$ is defined by
\[
\phi_*(h\langle w_i\rangle)=\phi_*(h)\langle u_j\rangle\,.
\]

A morphism of group pairs is said to be \emph{injective} if both the map of free groups and the map of sets is injective.  Perhaps surprisingly, this corresponds exactly to the notion of essential map of 2-complexes defined in the previous section. The following lemma \cite[Lemma 1.12]{wilton_essential_2018} characterises when $\phi$ is essential in terms of $\phi_*$.

\begin{lemma}\label{lem: Essential maps in group theory}
A map of finite, connected 2-complexes $\phi:Y\to X$ is essential if and only if the induced map of associated group pairs
\[
\phi_*:(F_Y,\curlyA_Y)\to (F_X,\curlyA_X)\,,
\]
is injective.
\end{lemma}

The lemma shows that, up to essential equivalence, essential maps to $X$ from finite, connected complexes $Y$ correspond naturally to subgroups $F_Y$ of $F_X$ and $F_Y$-invariant cofinite subsets $\curlyA_Y$ of $\curlyA_X$. 

\begin{remark}\label{rem: Composition of essential maps is essential}
It follows immediately from Lemma \ref{lem: Essential maps in group theory} that a composition of essential maps is essential.
\end{remark}

Two-complexes that are essentially equivalent  in the sense of \S\ref{sec: Branching and area} have, in particular, isomorphic group pairs. More generally, 2-complexes with isomorphic group pairs are said to be \emph{Nielsen equivalent}.

\begin{definition}\label{def: Nielsen equivalence}
An isomorphism of group pairs $(F,\curlyA)$ and $(H,\curlyB)$ consists of a group isomorphism $\phi:F\to H$ and a $\phi$-equivariant bijection $\curlyA\to\curlyB$. If the associated group pairs of 2-complexes $X$ and $Y$ are isomorphic, then $X$ and $Y$ are said to be \emph{Nielsen equivalent}.
 \end{definition} 

In particular, an essential equivalence of connected 2-complexes $f:Y\to X$ is a morphism of 2-complexes that induces an isomorphism between the associated group pairs.

For presentation complexes, Definition \ref{def: Nielsen equivalence} coincides with the usual definition of Nielsen equivalence of group presentations: $\langle a_1,\ldots,a_m\mid r_1,\ldots,r_n\rangle$ and $\langle a'_1,\ldots,a'_m\mid r'_1,\ldots,r'_n\rangle$ are Nielsen equivalent if and only if there is an isomorphism of free groups
\[
\alpha:\langle a_1,\ldots a_m\rangle\to \langle a'_1,\ldots a'_m\rangle
\]
such that, up to permuting indices, $\alpha(r_i)$ is conjugate to $r'_i$.

For branched 2-complexes, the area of the faces needs to be added to the data of the pair.

\begin{definition}[Branched group pair]\label{def: Branched group pair}
A \emph{branched group pair} consists of a group pair $(F,\curlyA)$ together with an $F$-invariant map $\Area:\curlyA\to \R_{\geq 0}$.  Two branched group pairs are said to be isomorphic if the underlying group pairs $(F,\curlyA)$ and $(H,\curlyB)$ are isomorphic and the associated bijection $\curlyA\to\curlyB$ preserves area.
\end{definition}

It is easy to see that a branched 2-complex determines a branched group pair, and so it makes sense to say that branched 2-complexes are Nielsen equivalent if their associated branched group pairs are isomorphic.

\section{Irreducible complexes}\label{sec: Irreducible complexes}

Two-dimensional complexes are too complicated to classify up to arbitrary homotopy equivalence. However, sometimes they can be simplified, and the following definition, which also played central roles in \cite{wilton_essential_2018} and \cite{louder_uniform_2024}, encapsulates various cases in which this is possible. The material in this section is very similar to \cite[\S3.1--2]{louder_uniform_2024}. Note that we do not assume that our 2-complexes are connected.

\begin{definition}[Visibly reducible and unfoldable]\label{def: Visibly reducible and unfoldable}
If a 2-complex $X$ satisfies any of the following three conditions then it is said to be \emph{visibly reducible}:
\begin{enumerate}[(i)]
\item the one-skeleton of $X$ has a vertex of valence 0 or 1;
\item a vertex of $X$ is locally separating;
\item $X$ has a \emph{free face}, i.e.\ an edge that is hit exactly once by one face.
\end{enumerate}
There is one further simplification move that plays an important role.
\begin{enumerate}[(i)]
\setcounter{enumi}{3}
\item  There is an essential equivalence $X'\to X$ that is not an isomorphism.
\end{enumerate}
If $X$ satisfies condition (iv) but none of conditions (i)-(iii) then it is said to be \emph{unfoldable}. If $X$ satisfies none of conditions (i)-(iv) then it is said to be \emph{visibly irreducible}.
\end{definition}

\begin{remark}\label{rem: Reducibility and links}
Each of conditions (i)-(iv) in Definition \ref{def: Visibly reducible and unfoldable} can be recognised from the structure of the links of the vertices of $X$:
\begin{enumerate}[(i)]
\item the link of some vertex has at most one vertex;
\item the link of some vertex is disconnected;
\item the link of some vertex contains a vertex of valence 1;
\item the link of some vertex has a cut vertex.
\end{enumerate}
In each case, there is a corresponding move that simplifies $X$.
\end{remark}

In summary, every 2-complex is visibly reducible, or unfoldable, or visibly irreducible.  Unfolding increases the number of vertices and edges while keeping the number of faces the same, and so it is not hard to see that a finite complex can only be unfolded finitely many times, at which point it will be either visibly reducible or visibly irreducible. Thus, the following definition applies to all 2-complexes.

\begin{definition}[Irreducible 2-complex]\label{def: Irreducible 2-complex}
A finite 2-complex $X$ is \emph{irreducible} if there is an essential equivalence $X'\to X$ such that $X'$ is visibly irreducible. Otherwise, $X$ is \emph{reducible}.
\end{definition}

Since it may be possible to unfold a 2-complex $X$ in several different ways, it is \emph{a priori} possible that $X$ could, on the one hand, be essentially equivalent to a visibly irreducible 2-complex $X'$ and, on the other hand, also be essentially equivalent to a visibly reducible 2-complex $X''$.  Using a famous lemma of Whitehead \cite{whitehead_equivalent_1936}, it transpires that irreducibility of a 2-complex can be characterised via algebraic properties of the associated group pair. We will use the following convenient formulation (cf.\ \cite[Proposition 3.8]{louder_uniform_2024}).

\begin{lemma}[Whitehead's lemma]\label{lem: Whitehead's lemma}
A finite 2-complex $X$ is reducible if and only if the associated group pair $(F,\curlyA)$ satisfies one of the following conditions:
\begin{enumerate}[(i)]
\item $F$ is trivial;
\item $F$ is cyclic and $\#\curlyA=1$;
\item there is a non-trivial free splitting $F=A*B$ and each stabiliser of an element of $\curlyA$ is conjugate into either $A$ or $B$.
\end{enumerate}
\end{lemma}

In particular, if $X$ is irreducible, then any unfolding sequence eventually terminates in a visibly irreducible 2-complex. 

Irreducible 2-complexes are very common. Indeed, even if $X$ is reducible, as long as $\pi_1(X)$ is sufficiently complicated, a procedure of repeatedly unfolding and simplifying will eventually produce a canonical irreducible 2-complex that maps to $X$. The following lemma is proved in \cite[Lemma 3.10 and 3.14]{louder_uniform_2024}.

\begin{lemma}[Irreducible cores]\label{lem: Irreducible cores}
Let $X$ be a connected, finite, standard 2-complex. If $\pi_1X$  is not free and doesn't split non-trivially as a free product, then there is an immersion $\phi:\widehat{X}\to X$ such that $\widehat{X}$ is irreducible and $\phi$ is an isomorphism of fundamental groups. Furthermore, any branched immersion from an irreducible branched complex $Y\to X$ factors canonically through $\widehat{X}$.
\end{lemma}

The complex $\widehat{X}$ is called the \emph{irreducible core} of $X$.  In fact, any finite complex $X$ can be factored into irreducible 2-complexes and simple discs, where a \emph{simple disc} is a 2-complex $D$ with a single face and homeomorphic to a disc. 

\begin{lemma}[Unfolding a 2-complex]\label{lem: Unfolding 2-complexes}
Let $X$ be any finite, connected, standard 2-complex. There is a finite, connected graph $G$, a finite set of connected, irreducible 2-complexes $X_1,\ldots,X_m$, and a finite set of discs $D_1,\ldots,D_n$ such that there is an essential equivalence
\[
X'= G\vee X_1\vee\ldots \vee X_m\vee D_1\vee\ldots\vee D_n \to X
\]
for some choice of iterated wedge $X'$ of the factors.
\end{lemma}

The proof is similar to the proofs of Lemmas \ref{lem: Whitehead's lemma} and \ref{lem: Irreducible cores}, and left as an exercise to the reader. The 2-complex $X'$ is called an \emph{unfolding} of $X$. We record one important consequence here. 

\begin{lemma}[Irreducible immersion]\label{lem: Irred is nonempty unless X is homotopic to a graph}
Let $X$ be a connected, finite, standard 2-complex. Unless $X$ is homotopy equivalent to a graph, there is an immersion $Y\to X$ with $Y$ irreducible.
\end{lemma}
\begin{proof}
Consider an unfolding $X'$ coming from Lemma \ref{lem: Unfolding 2-complexes}. If $m\geq 1$ then we may take $Y$ to be the result of folding the map  $X_1\to X$. Otherwise, $X'$ is a obtained by wedging discs onto the graph $G$, and so is homotopy equivalent to a graph, which implies that $X$ is too.
\end{proof}

Thus, as proposed in the introduction,  we can study very general classes of 2-complexes $X$ via maps from irreducible 2-complexes.  The reader should note, however, that even the very weak hypothesis of Lemma \ref{lem: Irred is nonempty unless X is homotopic to a graph} is not sharp, since there are irreducible 2-complexes that are homotopic to a point.

\begin{example}\label{eg: Contractible 2-complex}
Let $X$ be the presentation complex associated to the presentation $\langle a,b\mid bab^{-1}a^{-2},b\rangle$. Since this is a balanced presentation of the trivial group, it follows from the Hurewicz and Whitehead theorems that $X$ is contractible. But the link of the unique vertex of $X$ is the complete graph on 4 generators, so $X$ is irreducible by Remark \ref{rem: Reducibility and links}.
\end{example}

Example \ref{eg: Contractible 2-complex} is the first of an infinite family of examples of contractible, irreducible 2-complexes constructed by Miller and Schupp \cite{miller_presentations_1999}.

\section{Surfaces}\label{sec: Surfaces}

A finite 2-complex $X$ is said to be a \emph{surface} if its realisation is a closed topological surface. Equivalently, $X$ is a surface if and only if the link of every vertex is a circle. Since circles do not satisfy any of the conditions of Remark \ref{rem: Reducibility and links}, it follows that any surface is irreducible.  Surfaces provide some of the most important examples of irreducible complexes, and they will play a special role. 

Just as Whitehead's lemma characterises the group pairs associated to irreducible complexes, we will also need a result that characterises the 2-complexes that are Nielsen equivalent to surfaces. This is provided by the following lemma; the proof follows a theorem of Culler \cite{culler_using_1981}.

\begin{lemma}[Culler's lemma]\label{lem: Culler's lemma}
If $X$ is a finite 2-complex which is Nielsen equivalent to a surface $S$ then there is a  (homeomorphic, but possibly combinatorially distinct) surface $S'$ and an essential equivalence $S'\to X$.
\end{lemma}
\begin{proof}
Let $N$ be a closed tubular neighbourhood of the 1-skeleton of $S$ in its realisation. Then $N$ is homotopy equivalent to $S_{(1)}$ and the boundary components of $N$ correspond exactly to the faces of $S$. Because $X$ and $S$ are Nielsen equivalent, the fundamental groups of the 1-skeleta $S_{(1)}$ and $X_{(1)}$ are isomorphic, and because graphs are Eilenberg--Mac~Lane spaces, this isomorphism is induced by a homotopy equivalence $\phi:N\to X_{(1)}$.

Since the group isomorphism comes from a Nielsen equivalence it also respects the faces, and $\phi_*$ sends the conjugacy classes of the cyclic subgroups of $\pi_1(N)$ corresponding to the boundary components of $N$ bijectively to the conjugacy classes of the cyclic subgroups of $\pi_1(X_{(1)})$ corresponding to the faces of $X$.

We now follow the proof of \cite[Theorem 1.4]{culler_using_1981} to endow the realisation of $S$ with a new 2-complex structure, so that $\phi$ becomes a morphism of 2-complexes. After a homotopy, we may take $\phi$ to be transverse to the midpoints of edges of $X_{(1)}$. The preimages of these midpoints are then disjoint unions of properly embedded circles and arcs in $N$. Furthermore, the complementary regions map to neighbourhoods of vertices and so, since $\phi$ is a homotopy equivalence, every complementary region is a topological disc.

In particular, each embedded circle in the preimage bounds a disc, and so can be removed by a homotopy of $\phi$ which reduces the number of components of preimages. In the same way, we may also remove any arcs that are isotopic into the boundary of $N$.

After these modifications, the union of the preimages of the midpoints of edges of $X_{(1)}$ is a collection of properly embedded arcs, which cut $N$ into discs. The dual of this decomposition of $N$ is a graph $S'_{(1)}$, and $\phi$ naturally induces a morphism of graphs $\phi'$ from $S'_{(1)}$ to $X_{(1)}$. Furthermore, $\phi'$ is a homotopy equivalence because $\phi$ was. Finally, we view the natural maps from the boundary components of $N$ to $S'_{(1)}$ as attaching maps to define a 2-complex $S'$, and $\phi'$ extends to a morphism of 2-complexes $S'\to X$. We have already seen that $\phi'$ is a homotopy equivalence of 1-skeleta, and the construction also guarantees that $\phi'$ sends faces to faces bijectively, so $\phi'$ is in fact an essential equivalence.
\end{proof}

While Lemma \ref{lem: Irred is nonempty unless X is homotopic to a graph} implies that there exist essential maps from irreducible 2-complexes to a 2-complex $X$ under very mild hypotheses on $X$, it is far from obvious that there should be a similar result guaranteeing the existence of essential maps from surfaces to general $X$. However, such a result exists, and follows by combining Culler's lemma with \cite[Theorem F]{wilton_essential_2018}.

\begin{theorem}[Essential maps from surfaces exist]\label{thm: Essential surfaces exist}
Every finite, irreducible complex $X$ admits an essential branched map $S\to X$, where $S$ is a surface.
\end{theorem}

Surfaces are in some sense the simplest irreducible 2-complexes. The next theorem offers one way of making this heuristic precise.

\begin{theorem}[Surfaces are minimal among irreducible complexes]\label{thm: Surfaces are minimal}
A finite, irreducible 2-complex $X$ is a surface if and only if it satisfies the following property: every essential map $\phi:Y\to X$, with $Y$ finite and irreducible, is a finite-sheeted branched covering map.
\end{theorem}
\begin{proof}
We start with the `only if' direction. Suppose that $X$ is a surface, so every link of a vertex in $X$ is a circle. Consider an essential map $\phi:Y\to X$, where $Y$ is irreducible. If $\bar{\phi}:\bar{Y}\to X$ is the folded representative then $\bar{Y}$ is also irreducible, by Whitehead's lemma. Since the folded representative is a branched immersion, it induces injective maps on links; in particular, the link of each vertex of $\bar{Y}$ embeds into a circle. But every proper, non-empty subgraph of a circle has a vertex of valence either zero or one, which would imply that $\bar{Y}$ is visibly reducible, a contradiction. Therefore, $\bar{\phi}$ is surjective on links and hence a branched covering map. Finally, since circles do not have cut vertices, $\bar{Y}$ is not unfoldable, and hence $Y=\bar{Y}$.

To prove the `if' direction, suppose that $X$ satisfies the given property, so that any essential map $Y\to X$ is a finite-sheeted branched covering map if $Y$ is finite and irreducible. By Theorem \ref{thm: Essential surfaces exist}, there is an essential map $S\to X$ with $S$ of surface type. By the hypothesis the map $S\to X$ is a branched covering map, so the induced maps on links are isomorphisms. Thus every link of $X$ is also a circle, and so $X$ is also a surface.
\end{proof}

\section{Curvature invariants}\label{sec: Curvature invariants}

We are now ready to define our notions of curvature.  The first definitions adapt Euler characteristic to the context of branched complexes.

\begin{definition}[Total and average curvature]\label{def: Total and average curvature}
The \emph{total curvature} of a finite branched 2-complex $X$ is the quantity
\[
\tau(X):=\Area(X)+\chi(X_{(1)})
\]
where $\chi(X_{(1)})$ is the usual Euler characteristic of the 1-skeleton $X_{(1)}$, namely $\#V_X-\#E_X$. The \emph{average curvature} of $X$ is
\[
\kappa(X):=\frac{\tau(X)}{\Area(X)}=1+\frac{\chi(X_{(1)})}{\Area(X)}\,.
\]
\end{definition}

The total curvature is closely related to the Euler characteristic.

\begin{remark}\label{rem: Total curvature and Euler characteristic}
If $X$ is standard then $\tau(X)$ is the Euler characteristic $\chi(X)$. Furthermore, if $\phi:Y\to X$ is a branched map then each face $f$ has area $m_\phi(f)\geq 1$, and so $\tau(Y)\geq\chi(Y)$.
\end{remark}

On their own, the total and average curvatures cannot determine much about $X$. The idea of our next invariants is to extremise these quantities over complexes mapping to $X$.  

\begin{definition}[Irreducible curvature]\label{def: Irreducible curvatures}
Let $X$ be a finite branched 2-complex, and let $\Irred(X)$ denote the set of essential maps $Y\to X$, where $Y$ is any finite, irreducible, branched 2-complex.  The quantity
\[
\rho_+(X):=\sup_{Y\in\Irred(X)}\kappa(Y)
\]
is called the \emph{maximal irreducible curvature} of $X$. Likewise, the \emph{minimal irreducible curvature} $\rho_-(X)$ is defined to be the corresponding infimum. By convention, $\rho_\pm(X)=\mp\infty$ if $\Irred(X)$ is empty.
\end{definition}

We saw in \S\ref{sec: Irreducible complexes} that $\Irred(X)$ is usually non-empty; for instance, Lemma \ref{lem: Irred is nonempty unless X is homotopic to a graph} says that $\Irred(X)$ is non-empty unless $X$ is homotopy-equivalent to a graph.

Definition \ref{def: Irreducible curvatures} can be used to say that a 2-complex $X$ has \emph{non-positive} irreducible curvature (if $\rho_+(X)\leq 0$), \emph{negative} irreducible curvature (if $\rho_+(X)< 0$),  \emph{constant} irreducible curvature ($\rho_+(X)=\rho_-(X)$), and so on. Since geometric group theorists are traditionally most interested in upper bounds on curvature, we will be especially concerned with $\rho_+(X)$.

Instead of extremising over essential maps from the relatively large class of irreducible complexes, one may instead restrict attention to the smaller class of (closed) surfaces, and this leads to another pair of invariants in the same way.

\begin{definition}[Surface curvature]\label{def: Surface curvatures}
 Let $X$ be a finite branched 2-complex, and let $\Surf(X)$ denote the set of essential maps $Y\to X$, where $Y$ is a closed surface.  The quantity
\[
\sigma_+(X):=\sup_{Y\in\Surf(X)}\kappa(Y)
\]
is called the \emph{maximal surface curvature} of $X$. Likewise, the \emph{minimal surface curvature} $\sigma_-(X)$ is defined to be the corresponding infimum. We again adopt the convention that $\sigma_\pm(X)=\mp\infty$ if $\Surf(X)$ is empty.
\end{definition}

Since a composition of essential maps is essential, Lemma \ref{lem: Irred is nonempty unless X is homotopic to a graph} and Theorem \ref{thm: Essential surfaces exist} together guarantees that $\Surf(X)$ is non-empty unless $X$ is homotopy equivalent to a graph.

As before, Definition \ref{def: Surface curvatures} can be used to make sense of the ideas of a 2-complex $X$ having \emph{non-positive} surface curvature (if $\sigma_+(X)\leq 0$), \emph{negative} surface curvature (if $\rho_+(X)< 0$) etc., and we will be especially concerned with $\sigma_+(X)$, because of geometric group theory's traditional focus on upper curvature bounds.

\section{Inequalities and first examples}\label{sec: Inequalities and first examples}

We start with a few easy observations. Throughout this section we will assume that $X$ is an irreducible branched complex. However, the irreducibility hypothesis can usually be relaxed to the weaker assumption that $\Irred(X)$ is non-empty.

The first result records some easy inequalities between the invariants.

\begin{lemma}[Trivial inequalities]\label{lem: Trivial inequalities}
If $X$ is a finite, irreducible, branched 2-complex then 
\[
-\infty\leq\rho_-(X)\leq\sigma_-(X)\leq\sigma_+(X)\leq\rho_+(X)\leq 1\,.
\]
\end{lemma}
\begin{proof}
The first inequality, that $-\infty\leq \rho_-(X)$, is vacuous. Since surfaces are irreducible, there is a natural inclusion $\Surf(X)\subseteq\Irred(X)$, whence $\rho_-(X)\leq\sigma_-(X)$ and $\sigma_+(X)\leq\rho_+(X)$. By Theorem \ref{thm: Essential surfaces exist}, $\Surf(X)\neq\varnothing$, so $\sigma_-(X)\leq\sigma_+(X)$. Finally, for any essential $Y\to X$, we have
\[
\kappa(Y)=1+\frac{\chi(Y_{(1)})}{\Area(Y)}
\] 
by definition. If $Y$ is irreducible then in particular the 1-skeleton $Y_{(1)}$ is a core graph, so $\chi(Y_{(1)})\leq 0$ and  $\kappa(Y)\leq 1$. Taking the supremum over all $Y\in \Irred(X)$, the final inequality follows,
\end{proof}

Eventually, we will sharpen these inequalities to prove Theorem \ref{thm: Curvature inequalities}. The next lemma records that the invariants are monotonic in the obvious way under essential maps.

\begin{lemma}[Monotonicity]\label{lem: Monotonicity}
If $\psi:X'\to X$ is an essential branched morphism of finite, irreducible, branched 2-complexes then:
\begin{enumerate}[(i)]
\item $\rho_+(X')\leq\rho_+(X)$ and $\rho_-(X')\geq \rho_-(X)$;
\item $\sigma_+(X')\leq\sigma_+(X)$ and $\sigma_-(X')\geq \sigma_-(X)$.
\end{enumerate}
\end{lemma}
\begin{proof}
Since a composition of essential maps is essential, there is a natural map $\psi_*:\Irred(X')\to \Irred(X)$ obtained by composing an essential map $Y\to X$ with $\psi$. Since $\kappa(\psi_*(Y))=\kappa(Y)$, this proves item (i). Furthermore, $\psi_*(\Surf(X'))\subseteq\Surf(X)$, so item (ii) also follows.
\end{proof}

When the essential map in question is a branched covering, we can say more. The next lemma encodes the observation that average curvature is constant under passing to branched coverings. Note that the restriction of any branched covering map $\phi:X'\to X$ to 1-skeleta is a covering map of graphs. If $X$ is connected then this covering map has a degree, and we may define the \emph{degree} of $\phi$ to be the degree of the covering map of 1-skeleta. In particular, $\phi$ is said to be \emph{finite sheeted} if $\deg(\phi)$ is finite.

The next lemma shows that the quantities $\tau(X)$ and $\kappa(X)$ interact well with finite-sheeted branched covering maps.

\begin{lemma}[Finite-sheeted branched coverings]\label{lem: Branched coverings have equal average curvature}
If $\phi:X'\to X$ is a finite-sheeted branched covering map of finite, irreducible 2-complexes then $\tau(X')=\deg(\phi)\tau(X)$ and $\kappa(X')=\kappa(X)$.
\end{lemma}
\begin{proof}
The proof that  
\[
\tau(X')=\deg(\phi)\tau(X)
\]
is identical to the standard proof of the Riemann--Hurwitz theorem. Since $\Area(X')=\deg(\phi)\Area(X)$ trivially, it follows that  $\kappa(X')=\kappa(X)$.
\end{proof}

It follows that our curvature bounds $\rho_\pm$ and $\sigma_\pm$ are also invariant under passing to branched coverings.

\begin{proposition}[Branched coverings and curvature bounds]\label{prop: Invariance under branched coverings}
If $X'\to X$ is a finite-sheeted branched covering map of finite, irreducible 2-complexes then $\rho_\pm(X')=\rho_\pm(X)$ and $\sigma_\pm(X')=\sigma_\pm(X)$.
\end{proposition}
\begin{proof}
We prove the proposition for the irreducible curvatures; the proof for surface curvatures is identical. By Lemma \ref{lem: Monotonicity}, $\rho_-(X)\leq\rho_-(X')$ and $\rho_+(X)\geq\rho_+(X')$, so it suffices to prove the reverse inequalities. 

Suppose now that $Y\in\Irred(X)$. Without loss of generality, we may assume that $Y$ is visibly irreducible. Since branched immersions pull back along essential maps \cite[Lemma 3.13]{louder_uniform_2024},  there is a finite-sheeted branched covering $Y'\to Y$ and an essential map $Y'\to X'$. Since $Y'$ has the same links as the visibly irreducible complex $Y$, it is itself visibly irreducible, so $Y'\in\Irred(X')$. Now $\kappa(Y')=\kappa(Y)$ by Lemma \ref{lem: Branched coverings have equal average curvature} and so, taking suprema or infima as appropriate, the result follows.
\end{proof}

A little caution is needed when applying Proposition \ref{prop: Invariance under branched coverings}. If $X$ is standard then $X'$ is only standard if the map $X'\to X$ is an \emph{unbranched} covering.

The results discussed so far make it easy to compute the curvatures of our first family of examples.

\begin{example}[Surfaces]\label{eg: Surfaces}
Consider a closed surface $S$. By Theorem \ref{thm: Surfaces are minimal}, $Y\in\Irred(S)$ if and only if there is a finite-sheeted branched covering $Y\to S$, so $\kappa(Y)=\kappa(S)$ by Lemma \ref{lem: Branched coverings have equal average curvature}. Thus, by definition,
\[
\rho_\pm(S)=\kappa(S)
\]
and therefore $\sigma_\pm(S)=\kappa(S)$ by Lemma \ref{lem: Trivial inequalities}. In particular, if $S$ is standard and has one face, then all curvatures equal the Euler characteristic $\chi(S)$.
\end{example}

In summary, surfaces have constant irreducible curvature. In fact, Theorem \ref{thm: Constant irreducible curvature} tells us that the equation $\rho_+(S)=\rho_-(S)$ is surprisingly close to a characterisation of surfaces.

\section{Nielsen equivalence}\label{sec: Nielsen equivalence}

The curvature invariants defined here are too sensitive to be invariant under homotopy equivalence. The irreducible, contractible 2-complex $X$ of Example \ref{eg: Contractible 2-complex} provides an absurd example, since
\[
\rho_+(X)\geq\kappa(X)=1/2
\]
but $X$ is homotopy equivalent to the 1-point complex $*$, which has $\rho_+(X)=-\infty$.

Thus, $\rho_+$ is not invariant under homotopy equivalence, and indeed similar reasoning tells us that neither are $\rho_-$ or $\sigma_\pm$. However, they are invariant under a stronger equivalence relation. Recall from \S\ref{sec: Group pairs} that two 2-complexes are \emph{Nielsen equivalent} if their associated group pairs are isomorphic. The purpose of this section is to prove that the quantities $\rho_\pm(X)$ and $\sigma_\pm(X)$  are invariant under Nielsen equivalence. That is, we need to show that $\rho_\pm(X)$ and $\sigma_\pm(X)$ only depend on the associated group pair $(F_X,\curlyA_X)$. 

Since taking fundamental groups only makes good sense for connected complexes, the first observation is that we may assume that every $Y\in \Irred(X)$ is connected.

\begin{remark}\label{rem: Average of connected components}
Let $Y$ be a finite branched 2-complex of positive area. If $Y=Y_1\sqcup \ldots \sqcup Y_n$, then $\tau(Y)$ is the sum of the $\tau(Y_i)$, so
\[
\kappa(Y)=\sum_i t_i\kappa(Y_i)
\]
where $0\leq t_i= \Area(Y_i)/\Area(Y)\leq 1$. In particular, 
\[
\min_i\kappa(Y_i)\leq\kappa(Y)\leq\max_i\kappa(Y_i)
\]
and so, in the definitions of $\rho_\pm(X)$ and $\sigma_\pm(X)$, we may restrict attention to connected $Y$  in $\Irred(X)$ or $\Surf(X)$.
\end{remark} 

By Lemma \ref{lem: Essential maps in group theory}, the branched group pairs of connected branched complexes in $\Irred(X)$ correspond exactly to branched group pairs $(F_Y,\curlyA_Y)$ equipped with injective maps $\phi_*:(F_Y,\curlyA_Y)\to (F_X,\curlyA_X)$. The next step is therefore to compute $\kappa(Y)$ from the group pair $(F_Y,\curlyA_Y)$, which comes down to computing $\chi(Y_{(1)})$ and $\Area(Y)$.  The former is straightforward, since
\[
\chi(Y_{(1)})=1-\rk F_Y
\]
when $Y$ is connected.

For the latter, we need to associate an area to each $b\in \curlyA_Y$, which corresponds to some face $f$ of $Y$. The image $a=\phi_*(b)\in\curlyA_X$ has stabiliser conjugate to some $\langle w_j\rangle$ which in turn corresponds to the face $\phi(f)$ of $X$. After conjugating, we may assume that the stabiliser of $b$ is some $\langle u_i\rangle$ and that $\phi_*(u_i)\in \langle w_j\rangle$. The multiplicity of $f$ can now be computed as the index
\[
m_\phi(f)=|\langle w_i\rangle:\langle \phi_*(u_j)\rangle |
\]
so, by the compatibility between area and multiplicity, we may compute the area of $f$ as
\[
\Area(f)=|\langle w_i\rangle:\langle \phi_*(u_j)\rangle |\Area(\phi(f))\,.
\]
Since $\Area(Y)$ is the sum of the areas of its faces, this shows that $\Area(Y)$ can also be computed from the group pair.

Since Whitehead's lemma (Lemma \ref{lem: Whitehead's lemma}) characterises those group pairs $(F_Y,\curlyA_Y)$ that correspond to \emph{irreducible} complexes $Y$, it follows that $\rho_\pm(X)$ are determined by the group pair $(F,\curlyA)$  associated to $X$. Thus we have proved:

\begin{proposition}[Nielsen invariance of irreducible curvatures]\label{prop: Nielsen invariance of irreducible curvatures}
If $X$ and $X'$ are Nielsen equivalent 2-complexes, then $\rho_\pm(X)=\rho_\pm(X')$.
\end{proposition}

An identical argument substituting Culler's lemma for Whitehead's lemma shows that $\sigma_\pm(X)$ also only depends on the group pair, yielding the following:

\begin{proposition}[Nielsen invariance of surface curvatures]\label{prop: Nielsen invariance of surface curvatures}
If two finite, irreducible 2-complexes $X$ and $X'$ are Nielsen equivalent, then $\sigma_\pm(X)=\sigma_\pm(X')$.
\end{proposition}

\section{Irrigid complexes}\label{sec: Irrigid complexes}

This section is devoted to studying a class of 2-complexes that slightly generalise surfaces, and which played important roles in \cite{wilton_one-ended_2011} and \cite{wilton_essential_2018}.  Here we christen them \emph{irrigid complexes}. Two important theorems are proved: Theorem \ref{thm: Curvature inequalities}, specifically the fact that $\rho_-=\sigma_-$, and Theorem \ref{thm: Constant irreducible curvature}, the classification of complexes of constant irreducible curvature.

\begin{definition}\label{def: Surface decomposition and irrigid complexes}
 A \emph{surface decomposition} of a group pair $(F,\curlyA)$ is a graph of groups $\curlyG$ such that $\pi_1\curlyG=F$ of the following form. The edge groups are all infinite cyclic. The underlying graph is bipartite, with vertices of the following two types.
\begin{enumerate}[(i)]
\item \emph{Surface vertices} $v$ have an associated compact surface with boundary $\Sigma_v$, and the associated vertex group is identified with $\pi_1(\Sigma_v)$.
\item \emph{Cyclic vertices} have infinite-cyclic vertex groups.
\end{enumerate}
For a surface vertex $v$, the incident edge groups are identified isomorphically with the fundamental groups of the boundary components of $\Sigma_v$. Some subset of the cyclic vertices, called the \emph{peripheral} vertices, is identified with the $F$-orbits of $\curlyA$, and the vertex group of a peripheral cyclic vertex is the corresponding stabiliser (which is well-defined up to conjugacy).

A 2-complex $X$ is called \emph{irrigid} if the associated group pair of each connected component admits a surface decomposition.
\end{definition}

It is immediately clear that if $X$ is a surface then it is irrigid.  Baumslag--Solitar groups provide a more interesting class of examples.

\begin{example}\label{eg: BS groups are of surface type}
Let $X$ be the standard presentation complex of the Baumslag--Solitar group
\[
BS(m,n)=\langle a,b\mid ba^mb^{-1}a^{-n} \rangle
\]
and let $r$ denote the relator $ba^mb^{-1}a^{-n}$. The free group $a,b$ can be viewed as the fundamental group of the following space. Let $\Sigma$ be a 3-holed sphere with boundary components $c_0,c_1,r$, and let $C$ be a circle with fundamental group generated by a simple loop $a$.   Then construct a graph of spaces $\curlyX$ by identifying $c_0$ with $a^n$ and $c_1$ with $a^{-m}$. This gluing introduces a stable letter $b$ in the fundamental group of $\curlyX$, and the fundamental group of $\curlyX$ naturally has presentation
\[
\langle a,b,c_0,c_1,r\mid c_0c_1r\,,\,c_0a^{-n}\,,c_1ba^mb^{-1}\rangle\,.
\]
It follows that $\pi_1(\curlyX)$ is free on the generators $a,b$, while $c_0=a^n$, $c_1=ba^{-m}b^{-1}$, and $r=c_1^{-1}c_0^{-1}=ba^mb^{-1}a^{-n}$. Therefore, the graph-of-groups decomposition for $F$ implicit in the construction of $\curlyX$ is a surface decomposition of $F$ relative to $r$. Since the only non-cyclic vertex group is a surface, this shows that $X$ is irrigid.
\end{example}

In general, an algorithm of Cashen can be used to determine whether or not a given pair is irrigid \cite{cashen_splitting_2016}.

Although irrigid complexes are more general than surfaces, they can in practice be replaced by branched maps from surfaces.

\begin{lemma}\label{lem: Surface to surface type}
If $X$ is a finite, irreducible, irrigid complex, then there is a branched surface $S$ with $\kappa(S)=\kappa(X)$ and an essential map $S\to X$.
\end{lemma}
\begin{proof}[Sketch proof]
Working componentwise, we may assume that $X$ is connected, so the associated group pair $(F,\curlyA)$ has a surface decomposition $\curlyG$ by hypothesis. After replacing $X$ by a branched cover, we may assume that every edge attaching map of $\curlyG$ is an isomorphism by Marshall Hall's theorem \cite{hall_jr_subgroups_1949}.

For each cyclic vertex group of $\curlyG$, choose a cyclic ordering on the incident edges. Since $X$ is irreducible, each non-peripheral cyclic vertex has at least two incident edges. Therefore, each edge is distinct from its successor in the cyclic order. Consider the double cover $X'=X_1\sqcup X_2$ of $X$ that consists of the disjoint union of two copies of $X$, corresponding to two copies $\curlyG_1$ and $\curlyG_2$ of $\curlyG$. We now build a new graph of groups $\curlyH$ by deleting all of the cyclic vertices and, for each edge group of $\curlyG_1$, identifying it with the next edge group of $\curlyG_2$ in the relevant cyclic order.

The resulting graph of groups $\curlyH$ consists of surfaces glued along their boundary components and so, attaching 2-cells to each vertex group as they were attached in the two copies of $\curlyG$, it defines a group pair $(H,\curlyB)$ of surface type. By construction, there is an induced map of group pairs $(H,\curlyB)\to (F,\curlyA)$. This induced map can be seen to be essential by examining normal forms in the natural graph-of-groups decompositions, \footnote{This is where the fact that each edge group is distinct from its successor in the cyclic order is used.} so $(H,\curlyB)$ corresponds to an essential map $Y\to X$ for some 2-complex $Y$. Since $(H,\curlyB)$ is of surface type, there is a surface $S$ and elementary equivalence $S\to Y$ by Culler's lemma.

It remains to compute $\kappa(S)$. By Lemma \ref{lem: Branched coverings have equal average curvature}, $\kappa(X)=\kappa(X_1\sqcup X_2)$. Since $\curlyH$ is obtained from $\curlyG_1$ and $\curlyG_2$ by cutting and gluing infinite cyclic subgroups (which have Euler characteristic zero), it follows that $\chi(Y_{(1)})=2\chi(X_{(1)})$. Since each 2-cell of $X$ appears exactly twice in $X_1\sqcup X_2$ and hence twice in $Y$,  it follows that $\Area(Y)=2\Area(X)$. Hence $\kappa(Y)=\kappa(X)$ and the result follows since $S\to Y$ is an essential equivalence.
 \end{proof}

The next lemma is a restatement of \cite[Lemma 5.9]{wilton_essential_2018}, which in turn is based on the main argument of \cite{wilton_one-ended_2011}.

\begin{lemma}\label{lem: Minimal average curvature implies surface type}
Let $X$ be a finite, connected, irreducible, branched 2-complex. If $X$ is not irrigid then there is $Y\in\Irred(X)$ such that $\kappa(Y)<\kappa(X)$.
\end{lemma}

With these results in hand, we can now easily prove two of the theorems from the introduction.

\begin{proof}[Proof of Theorem \ref{thm: Curvature inequalities}]
In the light of Lemma \ref{lem: Trivial inequalities}, it suffices to prove that
\[
-\infty<\rho_-(X)=\sigma_-(X)\,.
\]
The rationality theorem tells us that $\rho_-(X)$ is realised by some $Y\in\Irred(X)$, so $\rho_-(X)=\kappa(Y)>-\infty$. Furthermore, $Y$ must be irrigid by Lemma \ref{lem: Minimal average curvature implies surface type}, so Lemma \ref{lem: Surface to surface type} provides $S\in\Surf(Y)$ such that $\kappa(S)=\kappa(Y)=\rho_-(X)$.  Now
\[
\sigma_-(X)\leq \sigma_-(Y)\leq\kappa(S)=\rho_-(X)
\]
by Lemma \ref{lem: Monotonicity}, and the result follows.
\end{proof}

Our second theorem is the classification of complexes with constant irreducible curvature. Precisely, it asserts that every such 2-complex is irrigid.\footnote{Strictly speaking, this is not quite a classification because the necessary condition is not sufficient: there are irrigid 2-complexes with $\rho_+(X)>\rho_-(X)$. But it is not difficult to determine which irrigid 2-complexes have constant irreducible curvature.}

\begin{proof}[Proof of Theorem \ref{thm: Constant irreducible curvature}]
If $X$ is not irrigid then Lemma \ref{lem: Minimal average curvature implies surface type} gives $Y\in\Irred(X)$ such that $\kappa(Y)<\kappa(X)$. In particular,
\[
\rho_-(X)\leq \kappa(Y)<\kappa(X)\leq\rho_+(X)
\]
which contradicts the hypothesis that $\rho_-(X)=\rho_+(X)$.
\end{proof}

Since it plays an important role, it is also worth mentioning that the results stated here also give a proof that $\Surf(X)$ is non-empty for irreducible $X$ (i.e.\ Theorem \ref{thm: Essential surfaces exist}), which is slightly different to the original proof \cite{wilton_essential_2018}.

\begin{proof}[Proof of Theorem \ref{thm: Essential surfaces exist}]
Let $X$ be finite and irreducible, so $\Irred(X)\neq\varnothing$ since $X\in\Irred(X)$. By the rationality theorem, there is $Y\in \Irred(X)$ such that $\kappa(Y)=\rho_-(X)$. By Lemma \ref{lem: Minimal average curvature implies surface type}, $Y$ is irrigid, so there is $S\in\Surf(Y)$ by Lemma \ref{lem: Surface to surface type}. The natural embedding of $\Surf(Y)\to \Surf(X)$ now shows that $\Surf(X)$ is also non-empty, as required.
\end{proof}

\section{Relations to other invariants}\label{sec: Relations to other invariants}

Although $\sigma_\pm$ and $\rho_\pm$ appear to be new invariants, they are related to various other invariants already in the literature. The purpose of this section is to explain some of these relationships.

\subsection{Deficiency}

The deficiency of a presentation is an alternative normalisation of the Euler characteristic, traditionally used in combinatorial group theory.

\begin{definition}[Deficiency]\label{def: Deficiency}
The \emph{deficiency} of a finite group presentation
\[
\langle a_1,\ldots, a_n\mid r_1,\ldots, r_m\rangle
\]
is the quantity $n-m$. Note that this is $\delta(X)=1-\chi(X)$, where $X$ is the associated presentation complex.  The \emph{deficiency} $\delta(G)$ of a group $G$ is the maximum deficiency of a presentation for that group.
\end{definition}

Presentations with deficiencies bounded below have been a popular topic throughout the history of combinatorial group theory.\footnote{The Baumslag--Pride theorem, that groups with $\delta(G)\geq 2$ are large, is perhaps the most notable result in this direction \cite{baumslag_groups_1978}.} One limitation is that no general theorems can be proved about the subgroups of groups with deficiency bounded below, since by taking free products any finitely presented group appears as a free factor in a presentation with arbitrarily large deficiency. However, upper bounds on irreducible curvature in turn impose a natural lower bound on the deficiencies of the freely indecomposable finitely presented subgroups of $G=\pi_1(X)$.

\begin{remark}\label{rem: Irreducible curvature and deficiency}
Let $X$ be a finite, standard 2-complex. The inclusion of any non-free, freely indecomposable, finitely presented subgroup $H$ of $G=\pi_1(X)$ is represented by an immersion $Y\to X$, where $Y$ is irreducible \cite[Lemma 4.2]{louder_one-relator_2020}. If $\rho_+(X)= 0$ then $\tau(Y)\leq 0$ for every such $Y$, and so
\[
\delta(H)\geq\delta(Y) = 1-\chi(Y)\geq 1-\tau(Y)\geq 1\,.
\]
Likewise, if $\rho_+(X)<0$ then $\tau(Y)<0$ for every such $Y$, and so $\delta(H)\geq 2$.
\end{remark}

This suggests that $\rho_+(X)$ should be thought of as a a `stabilised' version of deficiency. The 2-complexes with non-positive irreducible curvature provide good candidates for presentations of deficiency at least one about which one can hope to prove useful theorems.  Likewise, the 2-complexes with negative irreducible curvature provide good candidates for presentations of deficiency at least two to study. A few such theorems are collected in \S\ref{sec: Properties of groups with curvature bounds} below.

\subsection{Primitivity rank}\label{sec: Primitivity rank}

Primitivity rank is an invariant of words in free groups introduced by Puder \cite{puder_primitive_2014}. A word $w$ in a free group $F$ is \emph{primitive} if it is contained in a minimal generating set, and otherwise is said to be \emph{imprimitive}.

\begin{definition}[Primitivity rank]\label{def: Primitivity rank}
Let $F$ be a free group and $w\in F$ a non-trivial word. The quantity
\[
\pi(w):=\min\{\rk(H)\mid w\mathrm{\,imprimitive\,in\,}H\leq F\}
\]
is the \emph{primitivity rank} of $w$. By convention, if $w$ is primitive in $F$ (and hence in all its subgroups) then $\pi(w)=\infty$.
\end{definition}

Puder and Parzanchevski related the primitivity rank to the behaviour of random word maps on symmetric groups \cite{puder_measure_2015}, while Louder and the author related it to the subgroup structure of 1-relator groups \cite{louder_negative_2022}\footnote{See \S\ref{subsec: One-relator groups} for further discussion of one-relator groups.}.The connection with one-relator groups stems from the observation that the primitivity rank is also related to an extremal value of the total curvature, when restricted to immersed 2-complexes of area 1.

\begin{lemma}\label{lem: Topological point of view on primitivity rank}
Let $F$ be a free group, let $w\in F$ be a non-trivial element, and let $X$ be the natural presentation complex of the one-relator group $F/\llangle w\rrangle$. Then
\[
2-\pi(w)=\max\{ \tau(Y)\mid Y\in \Irred(X)\,,\,\Area(Y)=1\}
\]
\end{lemma}
\begin{proof}
If $H\leq F$ realises the minimum in the definition of $\pi(w)$ then the associated immersion of core graphs $Y_{(1)}\to X_{(1)}$ carries a lift $\hat{w}$ of $w$. Since $H$ cannot split freely relative to $w$, the complex $Y$ constructed by attaching a face along $\hat{w}$ to $Y_{(1)}$ is irreducible, by Whitehead's lemma.  Since
\[
\tau(Y)=2-\rk H=2-\pi(w)
\]
this establishes that the left hand side of the desired equation is no greater than the right hand side.

To prove the reverse inequality, it remains to show that any other $Z\in\Irred(X)$ with $\Area(Z)=1$ has $\tau(Z)\leq \tau(Y)$.   Indeed, the 1-skeleton $Z_{(1)}$ defines a subgroup $K$ of $F$ containing $w$ imprimitively, again by Whitehead's lemma. But $\tau(Z)>\tau(Y)$ implies that $\rk(K)<\rk(H)$, because $\Area(Z)=\Area(Y)=1$, contradicting the hypothesis that $H$ realises the minimum. 
\end{proof}

The lemma quickly implies a relationship between the primitivity rank and the maximal irreducible curvature.

\begin{lemma}[Primitivity rank and $\rho_+(X)$]\label{lem: Relationship between primitivity rank and maximal irreducible curvature}
If $X$ is the presentation complex of the one-relator group $F/\llangle w\rrangle$ then 
\[
2-\pi(w)\leq \rho_+(X)\,.
\]
\end{lemma}
\begin{proof}
If $\Area(Y)=1$ then $\kappa(Y)=\tau(Y)$, so Lemma \ref{lem: Topological point of view on primitivity rank} shows that $2-\pi(w)=\kappa(Y)$ for some $Y\in \Irred(X)$. The result is then immediate.
\end{proof}

This result provides useful estimates of $\rho_+$ because the primitivity rank can often be computed by hand for small examples; see, for instance, \cite[Example 1.4]{louder_negative_2022}.\footnote{More generally, algorithms to compute $\pi(w)$ are described in \cite[Appendix A]{puder_measure_2015} and \cite[Lemma 6.4]{louder_negative_2022}.}

Lemma \ref{lem: Topological point of view on primitivity rank} suggests that the quantity $\rho_+(X)$ can be thought of as stabilised version of primitivity rank, in analogy with stabilising commutator length to give stable commutator length.\footnote{See \S\ref{sec: Stable commutator length} for more on stable commutator length.} We therefore propose the following definition.

\begin{definition}\label{def: Stable primitivity rank}
Let $F$ be a free group and $w\in F$ be a non-trivial element. The \emph{stable primitivity rank} of $w$ is defined to be
\[
s\pi(w):=1-\rho_+(X)
\]
where $X$ is the natural presentation complex associated to the one-relator group  $F/\llangle w \rrangle$.\footnote{The reader may reasonably wonder why $s\pi(w)$ is defined to be $1-\rho_+(X)$, rather than $2-\rho_+(X)$. This is because $s\pi(w)$ should be thought of as a (suitably normalised) Euler characteristic of a graph, not of a 2-complex, so we should subtract the term coming from the face of $X$.}
\end{definition}

A much stronger improvement of Lemma \ref{lem: Topological point of view on primitivity rank} was proved by Louder and the author in \cite[Lemma 6.10]{louder_negative_2022}.

\begin{theorem}\label{thm: Primitivity rank is an absolute extremum}
If $X$ is the presentation complex of the one-relator group $F/\llangle w\rrangle$ then 
\[
2-\pi(w)=\max_{Y\in\Irred(X)} \tau(Y)\,.
\]
\end{theorem}

This theorem provides a useful method to prove that $\rho_+(X)$ is negative -- see, for instance, Examples  \ref{eg: Manning's example} below.

\subsection{Stable commutator length}\label{sec: Stable commutator length}

Let $w$ be a non-trivial element of a free group $F$ as in the previous section, but this time we assume that $w$ is in the commutator subgroup $[F,F]$.  Recall the standard definitions of commutator length and its stable cousin.

\begin{definition}[Commutator length and stable commutator length]\label{def: Commutator length and stable commutator length}
The \emph{commutator length} of $w$, $\cl(w)$, is the minimal positive integer $n$ such that $w$ is a product of $n$ commutators.  The quantity
\[
\scl(w):=\inf_{n\in\N}\frac{\cl(w^n)}{n}
\]
is the \emph{stable commutator length} of $w$. (In fact this makes sense for any group.)
\end{definition} 

There are two important results about stable commutator length in a free group $F$. The Duncan--Howie theorem asserts that $\scl(w)\geq 1/2$ unless $w$ is a proper power \cite{duncan_genus_1991}. Calegari's rationality theorem asserts that $\scl(w)\in\Q$, and furthermore there is an algorithm to compute it \cite{calegari_stable_2009}. In fact, this algorithm is effective, and has been implemented \cite{calegari_scallop_????}.

As usual, we may consider a standard presentation complex $X$ for $F/\llangle{w}\rrangle$, with a single face $f$.  The next lemma documents a relationship between the surface curvature of $X$ and the stable commutator length.

\begin{lemma}[Surface curvature and stable commutator length]\label{lem: Surface curvature and stable commutator length (one face)}
For $w\in[F,F]$ and $X$ the standard presentation complex of $F/\llangle{w}\rrangle$,
\[
1-2\scl(w)\leq \sigma_+(X)\,.
\]
\end{lemma}

The lemma will follow from the stronger Proposition \ref{prop: Surface curvature and scl} below. The point of Lemma \ref{lem: Surface curvature and stable commutator length (one face)} is that the quantity $1-2\scl(w)$ can be thought of as an \emph{oriented} version of the surface curvature for the complex $X$. In this oriented version the areas of the faces of surfaces mapping to $X$ are counted with signs. Therefore, the oriented area is necessarily smaller than the unoriented area, which leads to the stated inequality.

From this point of view, the Duncan--Howie theorem asserts that the oriented surface curvature of a one-relator complex is always non-positive (cf.\ the stronger Theorem \ref{thm: NI for one-relator groups} below), while Calegari's theorem asserts that the oriented surface curvature is rational and computable (cf. the rationality theorem \cite{wilton_rationality_2022}).

In fact, the above discussion can be generalised beyond the one-relator context, to an arbitrary oriented 2-complex.

\begin{definition}[Orientation]\label{def: Orientation}
An \emph{orientation} on a branched 2-complex $X$ is a cellular 2-cycle
\[
c  = \sum_{f\in F(X)} c_f.f
\]
such that $|c_f|=\Area(f)$ for each face $f$. An \emph{oriented 2-complex} is a 2-complex $X$ equipped with an orientation $c$.
\end{definition}

Since $c$ is a 2-cycle it defines a non-trivial element of $H_2(X,\R)$, and also a non-trivial element of the second relative homology of the 1-skeleton $X_{(1)}$ relative to the faces.  In particular, $\partial c$ is a boundary in the graph $X_{(1)}$, and so we can define the \emph{stable commutator length} of $\partial c$. We recall the definition below, and refer the reader to Calegari's monograph \cite{calegari_scl_2009} for an authoritative discussion.

\begin{definition}[Stable commutator length]\label{def: Stable commutator length}
Let $\mathrm{Admiss}(c)$ be the set of continuous maps of compact, oriented surfaces $S\to X_{(1)}$ (not necessarily connected, but without spherical components) such that
\[
\partial S \mapsto\deg(S)\partial c\,,
\]
where $\deg(S)$ is a positive real number. Then
\[
\scl(c):=\inf_{S\in\mathrm{Admiss}(c)}-\frac{\chi(S)}{2\deg(S)}\,.
\]
\end{definition}

In the case where $X$ is the presentation complex of a one-relator group $F/\llangle w\rrangle$, this coincides with the usual definition of stable commutator length given above \cite[Proposition 2.10]{calegari_scl_2009}. We can now phrase a relationship between stable commutator length and surface curvatures for all oriented 2-complexes.

\begin{proposition}\label{prop: Surface curvature and scl}
Let $c$ be an orientation on a branched 2-complex $X$. Then
\[
\sigma_-(X)\leq 1-\frac{2\scl(c)}{\Area(X)}\leq\sigma_+(X)
\]
\end{proposition}
\begin{proof}
By Calegari's rationality theorem, there is $S\in\mathrm{Admiss}(c)$ such that
\[
\scl(c)=-\frac{\chi(S)}{2\deg(S)}\,.
\]
Calegari also shows that $S$ can be taken to be \emph{monotone}, meaning that the orientation on $\partial S$ restricted from the orientation on $S$ coincides with the orientation on $\partial S$ pulled back from $c$ \cite[Proposition 2.13]{calegari_scl_2009}. By capping off the boundary components, the surface $S$ defines a closed surface $\widehat{S}\to X$ with $\Area(\widehat{S})=\deg(S)\Area(X)$. Note also that
\begin{eqnarray*}
\kappa(\widehat{S})&=&1+\frac{\chi(S)}{\Area(\widehat{S})}\\
&=&1-\frac{2\deg(S)\scl(c)}{\deg(S)\Area(X)}\\
&=&1-\frac{2\scl(c)}{\Area(X)}\,.
\end{eqnarray*}
Such a surface $S$ is called \emph{extremal} and, by a lemma of Calegari \cite[Lemma 2.7]{calegari_surface_2008}, $\widehat{S}$ is essential. In particular, $\widehat{S}\in \Surf(X)$, and so
\[
\sigma_-(X)\leq \kappa(\widehat{S})\leq\sigma_+(X)
\]
which gives the result.
\end{proof}

Lemma \ref{lem: Surface curvature and stable commutator length (one face)} follows immediately by specialising to the case of a single face of area 1. Again, the moral is that a suitable normalisation of stable commutator length can be thought of as an oriented version of the surface curvature.

Recall that the Duncan--Howie theorem tells us that every one-relator complex has non-positive oriented surface curvature.  It remains an open question whether or not every 2-boundary $c$ in a free group satisfies $\scl(c)\geq 1/2$. However, from our point of view, the natural threshold corresponding to non-positive oriented surface curvature for an oriented complex is
\[
\scl(c) \geq \frac{\Area(X)}{2}\,,
\]
and some oriented 2-complexes do indeed have positive oriented surface curvature: for instance, any cellulation of the 2-sphere provides an example.

Since stable commutator length can be computed in practice, Proposition \ref{prop: Surface curvature and scl} offers a practical way to obtain lower bounds on $\sigma_+$ and upper bounds on $\sigma_-$.

\section{Examples}\label{sec: Examples}

So far, surfaces have provided the only classes of examples that we have considered. The goal of this section is to collect a wider variety of examples, which illustrate the  various possibilities that can occur. In the first few sections, however, we collect techniques for computing the invariants. 

\subsection{Angle structures and \texorpdfstring{$CAT(\kappa)$}{CAT(k)} spaces}

A more traditional way to impose curvature restrictions on a 2-complex $X$ is to use an \emph{angle structure}, namely a path metric on the link of each vertex.  Often, this angle structure is induced from a path metric on the whole of $X$, where each edge of each link is assigned the angle of the corresponding corner. Gromov's link condition implies that a Euclidean (respectively, hyperbolic) polygonal metric is locally $CAT(0)$ (respectively, locally $CAT(-1)$) if and only if, in the induced angle structure, the link of every vertex has systole at least $2\pi$ \cite[Lemma II.5.6]{bridson_metric_1999}.

An angle structure also makes it possible to impose bounds on the average curvature, following ideas of Wise \cite{wise_sectional_2004}. The next definition can be compared to  \cite[Definition 2.2]{wise_sectional_2004}.

\begin{definition}\label{def: Curvature of a 2-cell}
Consider a finite, irreducible 2-complex $X$, and suppose that $X$ is equipped with an angle structure. Let $f$ be a face of $X$ which is an $n(f)$-gon, meaning that its boundary is a circle with $n(f)$ edges, and with total interior angle $\alpha(f)$. The \emph{curvature} of $f$ is the quantity
\[
c(f):=2\pi\Area(f)-\pi n(f)+\alpha(f)\,.
\]
We say that the angle structure is \emph{non-positively curved} (respectively, \emph{negatively curved}) if:
\begin{enumerate}[(i)]
\item the link of every vertex has systole at least $2\pi$; and 
\item $c(f)\leq 0$ for each face $f$ (respectively, $c(f)<0$ for each face $f$).
\end{enumerate}
\end{definition}

This definition of the curvature of a face becomes easier to interpret with some input from the Gauss--Bonnet theorem for geodesic polygons.

\begin{lemma}[Riemannian Gauss--Bonnet]\label{lem: Riemannian Gauss--Bonnet}
Consider a face $f$ with $\Area(f)=1$. If $f$ is endowed with the angle structure of a geodesic polygon $P$ in some Riemannian surface, then
\[
c(f)=\int_P \kappa_{\mathrm{Gauss}}.dA
\]
where $\kappa_{\mathrm{Gauss}}$ is the Gauss curvature.
\end{lemma} 
\begin{proof}
The Gauss--Bonnet theorem for geodesic polygons asserts that
\[
\int_P \kappa_{\mathrm{Gauss}}.dA=\alpha(f)-(n(f)-2)\pi
\]
and the result follows immediately by the definition of $c(f)$.
\end{proof}

In particular, the angle structures coming from locally $CAT(0)$ Euclidean structures are non-positively curved, and the angle structures coming from locally $CAT(-1)$ hyperbolic structures are negatively curved. The use of this definition is in the following estimate, a combinatorial version of the Gauss--Bonnet theorem (cf.\ \cite[p.\ 107]{ballmann_nonpositively_1996}, \cite[Theorem 4.6]{mccammond_fans_2002}, \cite[Theorem 2.3]{wise_sectional_2004}).

\begin{lemma}[Combinatorial Gauss--Bonnet]\label{lem: Combinatorial Gauss--Bonnet}
Let $S$ be a finite branched surface equipped with an angle structure. If the total angle around every vertex is at least $2\pi$ then
\[
2\pi\tau(S)\leq \sum_{f\in F_S} c(f)\,.
\]
In particular, if the angle structure is non-positively curved then $\tau(S)\leq 0$. If the angle structure is negatively curved then $\tau(S)<0$.
\end{lemma}
\begin{proof}
Since the link of each vertex has total angle at least $2\pi$, the number of vertices of $X$ is at most the sum of the interior angles across $X$ divided by $2\pi$. Hence
\begin{eqnarray*}
\sum_{f\in F_X} c(f) &=& \sum_{f\in F_X} 2\pi\Area(f)-\pi n(f)+\alpha(f) \\
& = & 2\pi\Area(X) - 2\pi \#E_X +  \sum_{f\in F_X}\alpha(f) \\
& \geq &  2\pi \Area(X) - 2\pi\#E_X + 2\pi \#V_X \\
& = & 2\pi\tau(X)
\end{eqnarray*}
as required.
\end{proof}

Furthermore, the curvature of a face interacts well with branched maps.

\begin{remark}\label{rem: Curvatures of 2-cells are multiplicative}
Suppose that $\phi:Y\to X$ is a branched map, and that $f$ is a face of $Y$. If the image $\phi(f)$ has $n$ sides then $f$ has $m_\phi(f)n$ sides; likewise, if $\alpha$ is the total interior angle of $\phi(f)$ then $m_\phi(f)\alpha$ is the total interior angle of $f$. Since area is also multiplicative, it follows that
\[
c(f)=m_\phi(f)c(\phi(f))\\
\]
by definition.
\end{remark}

As a result, the curvature of a face provides an upper bound on the surface curvature of a standard 2-complex.

\begin{proposition}[Face curvature bounds surface curvature]\label{prop: Face curvature bound}
Let $X$ be a standard 2-complex equipped with an angle structure. If the link of every vertex has systole at least $2\pi$ then
\[
2\pi\sigma_+(X)\leq\max_{f\in F_X} c(f)\,.
\]
\end{proposition}
\begin{proof}
Consider an essential map from a surface $S\to X$. Since $S\to X$ is essential, links of $S$ map to immersed loops in links of $X$, and in particular are of length at least $2\pi$. Applying Lemma \ref{lem: Combinatorial Gauss--Bonnet} and Remark \ref{rem: Curvatures of 2-cells are multiplicative}, we get
\begin{eqnarray*}
2\pi\tau(S)&\leq &\sum_{f\in F_S} c(f)\\
& = &\sum_{f\in F_S} m_\phi(f) c(\phi(f))\\
&\leq & \Area(S) \max_{f\in F_X} c(f)
\end{eqnarray*}
and the result follows.
\end{proof}

In particular, a locally $CAT(0)$ Euclidean 2-complex has non-positive surface curvature.

\begin{corollary}[Surface curvature of locally $CAT(0)$ 2-complexes]\label{cor: Surface curvature of CAT(0) 2-complexes}
If $X$ is a finite 2-complex equipped with a non-positively curved angle structure (for instance, if $X$ is a locally CAT(0) Euclidean complex), then $\sigma_+(X)\leq 0$.
\end{corollary}

Likewise, locally $CAT(-1)$ hyperbolic 2-complexes have negative surface curvature.

\begin{corollary}[Surface curvature of locally $CAT(-1)$ 2-complexes]\label{cor: Surface curvature of CAT(-1) 2-complexes}
If $X$ is a finite 2-complex equipped with a negatively curved angle structure (for instance, if $X$ is a locally CAT(-1) hyperbolic complex), then $\sigma_+(X)< 0$.
\end{corollary}

The proposition has one more consequence which is very useful in practice, because it depends only on combinatorial information. By the \emph{girth} of a graph we mean the combinatorial length of the shortest non-trivial loop.

\begin{corollary}[Girths of links]\label{cor: Surface curvature and girths of links}
Let $X$ be a standard 2-complex in which every face has $n$ sides. Then
\[
\sigma_+(X)\leq 1-\frac{n}{2}\left(1-\frac{2}{\lambda}\right)
\]
where $\lambda$ is the smallest girth of the link of a vertex in $X$.
\end{corollary}
\begin{proof}
Equip $X$ with an angle structure in which every edge of every link has length $2\pi/\lambda$. Then the systole of every link is at least $2\pi$, and every face has curvature
\[
c(f)=2\pi-\pi n + 2\pi \frac{n}{\lambda}=2\pi\left( 1-\frac{n}{2}\left(1-\frac{2}{\lambda}\right) \right)
\]
and the result follows from Proposition \ref{prop: Face curvature bound}.
\end{proof}

\begin{remark}[Wise's sectional curvature]\label{rem: Wise's sectional curvature}
In Wise's terminology, the 2-complexes considered in this section are of non-positive or negative \emph{planar sectional curvature}. In \cite[Definition 2.4]{wise_sectional_2004}, Wise defines a stronger notion of (non-planar) \emph{sectional curvature}. In our terminology, his arguments show that a 2-complex $X$ with an angle structure of non-positive sectional curvature has $\rho_+(X)\leq 0$, while an angle structure of negative sectional curvature would guarantee that $\rho_+(X)< 0$.
\end{remark}

\subsection{Small-cancellation and random complexes}

Small-cancellation conditions are a useful combinatorial method of imposing curvature-like constraints on group presentations; see the book of Lyndon and Schupp for an authoritative treatment \cite[Chapter V]{lyndon_combinatorial_1977}.  

Informally, a \emph{piece} of $X$ is a path that appears in two different ways in the attaching maps of the 2-cells. Formally, it may be defined as follows. Let $S_X$ be the disjoint union of the boundaries of the 2-cells of $X$, so $X$ is characterised by the attaching map $S_X\to X_{(1)}$. A piece is then an off-diagonal component of the fibre product $S_X\times_{X_{(1)}} S_X$. A 2-complex satisfies the $C(p)$ \emph{small-cancellation condition} if no face of $X$ is a union of fewer than $p$ pieces. It is well known that the $C(6)$ condition can be thought of as a form of non-positive curvature, while the $C(7)$ condition can be thought of as a form of negative curvature. These heuristics fit precisely into the framework of surface curvature.

\begin{proposition}[$C(6)$ complexes have non-positive surface curvature]\label{prop: Surface curvature of C(6) complexes}
If $X$ is a finite, irreducible, standard, $C(6)$ complex then $\sigma_+(X)\leq 0$.
\end{proposition}
\begin{proof}
It suffices to prove that
\[
\tau(Y)\leq 0
\]
whenever $\phi:Y\to X$ is an essential map from a surface.

For a face $f$ of $Y$, let $n(f)$ be the number of vertices of $Y_{(1)}$ adjacent to $f$ of valence at least 3. Consider the natural graph structure on $Y_{(1)}$ obtained by deleting the vertices of valence 2.  The number of edges is then exactly
\[
\sum_{f\in F_Y} n(f)/2
\]
while the number of vertices is at most
\[
\sum_{f\in F_Y} n(f)/3\,,
\]
so 
\[
\tau(Y) \leq \sum_{f\in F_Y} \Area(f)-n(f)/2+n(f)/3 = \sum_{f\in F_Y} \Area(f) - n(f)/6\,.
\]
But the area of $f\in F_Y$ is $m_\phi(f)$ because $X$ is standard, and $n(f)\geq 6m_\phi(f)$ because $X$ is $C(6)$, so 
\[
\tau(Y)\leq \sum_{f\in F_Y} m_\phi(1-6/6)=0
\]
as required.
\end{proof}

In exactly the same way, we obtain a similar result for $C(7)$ complexes.

\begin{proposition}[$C(7)$ complexes have negative surface curvature]\label{prop: Surface curvature of C(7) complexes}
If $X$ is a finite, irreducible, standard $C(7)$ complex then $\sigma_+(X)< 0$.
\end{proposition}

Any $C(7)$ presentation complex $X$ with at least as many relators as generators now provides an example with $\sigma_+(X)<0$ but $\rho_+(X)>0$. For an especially large class of examples, one can use random presentations in the few-relator model. 

\begin{example}[Random few-relator complexes]\label{eg: Few-relator complexes}
Consider a randomly chosen presentation in the few-relator model with $n$ generators and $m$ relators of length $l$. The probability that this presentation is $C(7)$ tends to 1 as $l\to \infty$, and so the standard presentation complex also satisfies $\sigma_+(X)<0$ with high probability. On the other hand, $\chi(X)=1-n+m$ and the probability that $X$ is not visibly irreducible tends to 0 as $l\to\infty$. Therefore, with probability tending to 1,
\[
\rho_+(X)\geq\frac{\tau(X)}{\Area(X)}=\frac{1-n+m}{m}= 1+(1-n)/m>0
\]
unless $m<n$.
\end{example}

This provides a rich class of examples of complexes with negative surface curvature but positive irreducible curvature.

Since the inequality used in Example \ref{eg: Few-relator complexes} only provides a lower bound, it remains an interesting problem to estimate the irreducible curvatures of few-relator complexes with $m<n$. The following conjecture represents the most optimistic possibility.

\begin{conjecture}[Irreducible curvature bounds of few-relator complexes]\label{conj: Irreducible curvature bounds of few-relator complexes}
If $X$ is a random few-relator complex as in Example \ref{eg: Few-relator complexes}, then the probability that $\rho_+(X)=\kappa(X)$ tends to $1$ as $l\to \infty$. In particular, with high probability, $\rho_+(X)<0$ if $m<n-1$ and $\rho_+(X)=0$ if $m=n-1$.
\end{conjecture}

The \emph{density model} of random groups, in which $m$ is proportional to $(2n-1)^{dl}$ for some density $0<d<1$, is also much-studied, and the resulting group is infinite and hyperbolic with high probability if $d<1/2$ \cite{gromov_asymptotic_1993}. These presentation complexes  are irreducible and have positive Euler characteristic by definition, so certainly $\rho_+(X)>0$.  However,  a random complex satisfies $C(7)$ with high probability at density $d<1/12$ \cite{ollivier_small-cancellation_2007}, so such a complex $X$ does have negative surface curvature, and it is likely that this estimate extends to densities less than $1/2$.

\begin{conjecture}[Surface curvature of random complexes in the density model]\label{conj: Surface curvature of random complexes in the density model}
For a random presentation complex $X$ in the density model at density $d<1/2$, the probability that $\sigma_+(X)<0$ tends to $1$ as $l\to\infty$.
\end{conjecture}

\subsection{One-relator groups}\label{subsec: One-relator groups}

Groups with a single defining relation were much studied in the early years of combinatorial group theory. According to Chandler and Magnus, Dehn hypothesised a geometric solution to their word problem \cite[p.\ 114]{chandler_history_1982}.  However, Magnus' eventual solution was combinatorial \cite{magnus_identitatsproblem_1932}, and indeed none of the curvature conditions most commonly used in geometric group theory -- Gromov-hyperbolicity, automaticity, $CAT(0)$ metrics etc -- can apply to all one-relator groups, since they all imply a quadratic isoperimetric inequality.

In 2003, Wise introduced notions of \emph{non-positive} and \emph{negative immersions} \cite{wise_nonpositive_2003,wise_sectional_2004} as tools to attack certain open problems about one-relator groups, especially a famous question, due to Baumslag, asking whether or not one-relator groups are coherent \cite{baumslag_problems_1974}.\footnote{See \cite{wise_invitation_2020} for a survey of recent work on coherence.} These ideas have been developed over the last decade in papers by Helfer--Wise \cite{helfer_counting_2016} and Louder and the author \cite{louder_stackings_2017,louder_negative_2022}, culminating in resolutions of Baumslag's question for one-relator groups with torsion \cite{louder_one-relator_2020,wise_coherence_2022} and 2-free one-relator groups \cite{louder_uniform_2024}.  Very recently, Baumslag's question was completely resolved in the affirmative by Jaikin-Zapirain--Linton \cite{jaikin-zapirain_coherence_2023}.

These developments can be interpreted as the beginning of a geometric theory of one-relator groups. In this section, we will see that this new geometric theory of one-relator groups can also be interpreted using maximal irreducible curvature.

Consider the standard presentation complex $X$ for the presentation
\[
\langle a_1,\ldots,a_n\mid w\rangle
\]
of a one-relator group $G$. Recall that $G$ is torsion-free if and only if the relator $w$ is not a proper power \cite[Proposition 5.17]{lyndon_combinatorial_1977}.   Note also that, as long as $n>1$, the standard presentation complex $X$ is irreducible if and only if $w$ is not conjugate into a proper free factor.

The following theorem, proved independently by Helfer--Wise and by Louder and the author, can be thought of as the first proof of an upper bound on irreducible curvature \cite{helfer_counting_2016,louder_stackings_2017}. 

\begin{theorem}[Non-positive irreducible curvature for one-relator groups]\label{thm: Irreducible curvature of one-relator groups}
Let $X$ be the standard presentation complex of a one-relator group $G$, and suppose that $X$ is irreducible. If $G$ is torsion-free then $\rho_+(X)\leq 0$.
\end{theorem}

The theorem gives an exact computation of $\rho_+$ for two-generator, one-relator groups.

\begin{corollary}\label{cor: Irreducible curvature of two-generator, one-relator groups}
Let $X$ be the standard presentation complex of a two-generator, one-relator group. If $X$ is irreducible and $G$ is torsion-free then $\rho_+(X)=0$.
\end{corollary}
\begin{proof}
Since the presentation has two generators,
\[
\tau(X)=\chi(X)=2-2=0
\]
so the identity map realises the upper bound from the theorem. 
\end{proof}

Theorem \ref{thm: Irreducible curvature of one-relator groups} also implies that $\sigma_+(X)\leq 0$ by Lemma \ref{lem: Trivial inequalities}. This result contravenes the standard heuristics of geometric group theory. For instance, it is usually posited that any reasonable notion of non-positive curvature for a group $G$ should imply that $G$ satisfies a quadratic isoperimetric inequality.  However, there are famously Baumslag--Solitar groups with exponential Dehn function \cite[Exercise 7.2.11]{bridson_geometry_2002}, and yet their presentation complexes are non-positively curved in our sense.

\begin{example}[Baumslag--Solitar groups]\label{eg: Baumslag-Solitar groups}
The standard presentation complex $X$ of the $(m,n)$-Baumslag--Solitar group
\[
BS(m,n)\cong\langle a,b\mid b a^mb^{-1}a^{-n}\rangle
\]
has $\rho_+(X)=0$ by Corollary \ref{cor: Irreducible curvature of two-generator, one-relator groups}. Furthermore, $X$ is irrigid by Example \ref{eg: BS groups are of surface type}, so Lemma \ref{lem: Surface to surface type} provides $S\in\Surf(X)$ with $\kappa(S)=\kappa(X)=0$, whence $\sigma_+(X)=0$ also.
\end{example}

One-relator groups include many examples that are even more pathological, such as the Baumslag--Gersten group $BG=\langle a,b\mid b^a a(b^a)^{-1}a^{-2}\rangle$ \cite{baumslag_non-cyclic_1969,gersten_dehn_1992}.  From our perspective, all of these are non-positively curved as long as they are torsion-free. (The case with torsion is addressed in Theorem \ref{thm: Negative irreducible curvature for one-relator groups with torsion} below.)

The question of negative curvature for one-relator groups is considerably more delicate.  In Theorem  \ref{thm: Subgroups of groups with negative irreducible curvature} below, we shall see that if $\rho_+(X)<0$ and $G=\pi_1(X)$ then every 2-generator subgroup of $G$ is free. The results of Louder and the author \cite{louder_negative_2022,louder_uniform_2024} show that this is the only obstruction to negative irreducible curvature for one-relator groups.

\begin{theorem}[Torsion-free one-relator groups with negative irreducible curvature]\label{thm: NI for one-relator groups}
If $X$ is the standard presentation complex of a torsion-free one-relator group $G$ and every two-generator, one-relator subgroup of $G$ is free then $\rho_+(X)< 0$.
\end{theorem}

In \cite{louder_negative_2022}, these one-relator groups are said to have \emph{negative immersions}. 

The above theorems apply to one-relator groups without torsion, but  the case with torsion is also of great interest. The following account reframes the proof of coherence given in \cite{louder_one-relator_2020} (similar to Wise's independent proof \cite{wise_coherence_2022}) in the language of irreducible curvature.

Consider a one-relator presentation
\[
G\cong\langle a_1,\ldots, a_n\mid u^k\rangle
\]
where $k>1$. The obvious presentation complex $X$ fails to be concise, and so $\rho_+(X)=1$ by Proposition \ref{prop: rho_+=1}, showing that the hypothesis that $G$ should be torsion-free in Theorem \ref{thm: Irreducible curvature of one-relator groups} cannot be removed. However, there is a natural modification of the ideas of this paper in which $X$ is allowed to be an \emph{orbi-complex}, and the face of $X$ can be endowed with a cone point. In this modified framework, $G=\pi_1(X)$, where the single face of $X$ is attached along the word $u$ and has a cone point of order $k$. The standard assignment of areas should give the unique face of $X$ an area of $1/k$. 

In this modified framework, the techniques of the proof of Theorem \ref{thm: Irreducible curvature of one-relator groups} apply as before, except that the areas of faces are divided by $k$.\footnote{Since one-relator groups are virtually torsion-free, one may also pass to a finite-sheeted cover which is a genuine complex, as opposed to an orbi-complex, as in \cite[Theorem 2.2]{louder_one-relator_2020}. This avoids the need to discuss orbi-complexes.} For any essential map from an irreducible complex $Y\to X$, the proof of Theorem \ref{thm: Irreducible curvature of one-relator groups} now gives
\[
k\Area(Y)+\chi(Y_{(1)})\leq 0\,.
\]
Since the left hand side can be rewritten as $(k-1)\Area(Y)+\tau(Y)$, it follows that
\[
\kappa(Y)=\frac{\tau (Y)}{\Area(Y)}\leq 1-k\,.
\]
In conclusion we have the following estimate on the irreducible curvature of the presentation orbi-complex of a one-relator group with torsion, which was implicit in both \cite{louder_one-relator_2020} and \cite{wise_coherence_2022}.  

\begin{theorem}[Negative irreducible curvature for one-relator groups with torsion]\label{thm: Negative irreducible curvature for one-relator groups with torsion}
If $X$ is the standard presentation orbi-complex for a one-relator group
\[
G\cong\langle a_1,\ldots, a_n\mid u^k\rangle
\]
then $\rho_+(X)\leq 1-k$. In particular, if $k>1$ then $\rho_+(X)<0$.
\end{theorem}

This is in line with the general expectation that one-relator groups with torsion are negatively curved. For instance, Newman's spelling theorem asserts that all one-relator groups with torsion are hyperbolic \cite{newman_some_1968}.

Finally, we make a remark about one-relator groups and Nielsen equivalence. Propositions \ref{prop: Nielsen invariance of irreducible curvatures} and \ref{prop: Nielsen invariance of surface curvatures} show that $\rho_\pm$ and $\sigma_\pm$ are always invariant under Nielsen equivalence of \emph{presentations}. In the special case of one-relator groups, this can be extended to Nielsen equivalence of generating sets by a deep theorem of Magnus \cite[Proposition 5.8]{lyndon_combinatorial_1977}.

\begin{theorem}[Magnus' theorem]\label{thm: Magnus' theorem}
Consider two one-relator presentations
\[
\langle a_1,\ldots, a_m\mid w\rangle~\textrm{and}~\langle a'_1,\ldots, a'_m\mid w'\rangle
\]
for a group. If $\{a_1,\ldots, a_m\}$ and $\{a'_1,\ldots, a'_m\}$ are Nielsen equivalent as generating sets then the two presentations are also Nielsen equivalent.
\end{theorem}

This raises the possibility that our invariants could, at least in principle, be used to distinguish Nielsen equivalence classes of generating sets of one-relator groups.

\begin{corollary}[Nielsen equivalence of generating sets for one-relator groups]
Consider two one-relator presentations
\[
\langle a_1,\ldots, a_m\mid w\rangle~\textrm{and}~\langle a'_1,\ldots, a'_m\mid w'\rangle
\]
for a group, with corresponding presentation complexes $X$ and $X'$ respectively. If $\{a_1,\ldots, a_m\}$ and $\{a'_1,\ldots, a'_m\}$ are Nielsen equivalent as generating sets then $\rho_\pm(X)=\rho_\pm(X')$ and $\sigma_\pm(X)=\sigma_\pm(X')$.
\end{corollary}

\subsection{Geometric 2-complexes}

The next definition allows us to apply these ideas to 3-manifolds. 

\begin{definition}[Geometric 2-complex]\label{def: 3-manifold spine}
A  2-complex $X$ is called \emph{geometric} if it is endowed with the following additional structure.
\begin{enumerate}[(i)]
\item For each vertex $v$, the link $\Lk(v)$ is equipped with a planar embedding, i.e.\ a cellular embedding into $S^2$. This determines how to canonically thicken the 1-skeleton $X_{(1)}$ to a (possibly non-orientable) handlebody.
\item For each face $f$, the corresponding curve on the surface of the handlebody is 2-sided.
\end{enumerate}
In particular, a presentation is called \emph{geometric} if its presentation complex is geometric.
\end{definition}

By thickening the 1-skeleton to a handlebody and the faces to 2-handles, a geometric 2-complex $X$ determines a 3-manifold with boundary $M=M(X)$ that deformation retracts to $X$. Note that the structure of a geometric complex pulls back along an essential map.

If $X$ is an irreducible geometric 2-complex then the boundary $\partial M$ of the thickened 3-manifold is naturally endowed with the structure of a 2-complex, and the inclusion map induces an immersion $\partial M\to X$. The next fact  is morally a consequence of  Poincar\'e duality, which implies that $\chi(\partial M)=2\chi(M)$ for any compact 3-manifold, although we give a direct proof using the spherical links of $X$ (cf.\ \cite[Theorem 10.2]{wise_sectional_2004}).

\begin{proposition}[Poincar\'e duality]\label{prop: Poincar\'e duality}
If $X$ is a finite, irreducible, geometric, branched 2-complex then $\kappa(\partial M(X))=\kappa(X)$.
\end{proposition}
\begin{proof}
Every face of $X$ appears twice in $\partial M$, so $\Area(\partial M(X))=2\Area(X)$, and it suffices to prove that $\chi(\partial M_{(1)})=2\chi(X_{(1)})$.  To compute $\chi(\partial M_{(1)})$, we introduce an auxiliary 2-complex $L$, the disjoint union of the cellular decompositions of the 2-spheres given by the links of the vertices of $X$.

Since $L$ consists of one sphere for each vertex of $X$, we have $\chi(L)=2\#V_X$, while $2\#E_X=\#V_L$ by definition. Therefore
\begin{eqnarray*}
2\chi(X_{(1)})&=&2\#V_X - 2\#E_X \\
&=& \chi(L)-\#V_L\\
&=&\#F_L-\#E_L\,.
\end{eqnarray*}
But $F_L$ naturally bijects with $V_{\partial M}$, while $E_{\partial M}$ bijects with $E_L$, so $2\chi(X_{(1)})=\chi(\partial M_{(1)})$ as required.
\end{proof}

This in turn implies that the irreducible curvature of a 3-manifold spine is always equal to the surface curvature.

\begin{corollary}[Curvatures of geometric complexes]\label{cor: Curvatures of 3-manifold spines}
If $X$ is a geometric 2-complex then $\rho_\pm(X)=\sigma_\pm(X)$.
\end{corollary}
\begin{proof}
For any $Y\in\Irred(X)$, the proposition implies that $\kappa(\partial M(Y))=\kappa(Y)$. Since $\partial M(Y)\to Y\to X$ is an essential map from a surface, it follows that the range of values attained by maps from irreducible 2-complexes is equal to the range of values attained by maps from surfaces.
\end{proof}

Some 2-complexes are not geometric but nevertheless are \emph{virtually} geometric, meaning that some finite-sheeted branched cover is geometric \cite[Theorem 21]{gordon_surface_2010}. Virtually geometric 2-complexes were classified by Cashen \cite[Theorem 5.9]{cashen_splitting_2016}. Since Corollary \ref{cor: Curvatures of 3-manifold spines} applies to branched 2-complexes as well as standard ones, it follows that virtually geometric 2-complexes also satisfy this conclusion.

\begin{corollary}[Curvatures of virtually geometric 2-complexes]\label{cor: Curvatures of virtually geometric 2-complexes}
If $X$ is a virtually geometric 2-complex then $\rho_\pm(X)=\sigma_\pm(X)$.
\end{corollary} 
\begin{proof}
If $X'\to X$ is a geometric branched covering, then $\rho_\pm(X')=\sigma_\pm(X')$ by Corollary \ref{cor: Curvatures of 3-manifold spines}, so
\[
\rho_\pm(X)=\rho_\pm(X')=\sigma_\pm(X')=\sigma_\pm(X)
\]
by Proposition \ref{prop: Invariance under branched coverings}.
\end{proof}

Conjecturally, Corollary \ref{cor: Curvatures of 3-manifold spines} should explain many of the unusually good properties of 3-manifold groups, such as coherence \cite{scott_finitely_1973} (cf.\ Conjecture \ref{conj: Non-positive irreducible curvature and coherence} below).  In order to apply it, however, one would need bounds on the surface curvatures of geometric 2-complexes. These bounds are the content of the next two conjectures. The first is a cousin of Papakyriakopoulos' sphere theorem \cite{papakyriakopoulos_dehn_1957}.

\begin{conjecture}[Non-positive curvature for geometric complexes]\label{conj: Non-positive curvature for 3-manifolds}
Let $X$ be a finite, irreducible, geometric 2-complex. If $\sigma_+ (X)>0$ then either $\partial M(X)$ has a spherical component or $\pi_1(X)$ splits freely.
\end{conjecture}

The analogous conjecture for negative curvature can be thought of as a weak form of hyperbolisation for Haken 3-manifolds.

\begin{conjecture}[Negative curvature for geometric complexes]\label{conj: Negative curvature for 3-manifolds}
Let $X$ be a finite, irreducible, geometric, branched 2-complex. If $\sigma_+ (X)=0$ then either $\partial M(X)$ has a toroidal boundary component or $\pi_1(X)$ has a $\Z^2$ subgroup.
\end{conjecture}

\subsection{Graphs of free groups}\label{sec: Graphs of free groups}

Graphs of free groups have been a popular class of examples throughout the history of geometric group theory. They exhibit a rich range of behaviours and are very difficult to study in full generality. In the same spirit, I expect them to be an interesting test class for our rational curvature invariants. Nothing is currently known, so we will state some conjectures and questions.

Before doing so, however, we should note that there are various possible choices of 2-complex that one might associate to a given graph of free groups. Nielsen invariance provides some reassurance, but nevertheless the values of $\rho_\pm$ and $\sigma_\pm$ will depend on this choice. Making the correct choice of 2-complex will necessarily be part of resolving the conjectures and questions below.

Mapping tori of free group automorphisms and, more generally, endomorphisms, have proved an especially fruitful class of examples. The following conjecture posits that they should all be non-positively curved in our strongest sense.

\begin{conjecture}[Irreducible curvatures of mapping tori of free group endormorphisms]\label{conj: Irreducible curvature of mapping tori of free group endomorphisms}
If $X$ is a compact mapping torus of an injective endomorphism of a finitely generated free group then $\rho_+(X)= 0$.
\end{conjecture}

The lower bound $\rho_+(X)\geq 0$ follows from the fact that $\chi(X)=0$. Wise proved that such $X$ have the slightly weaker \emph{non-positive immersions} property \cite[Theorem 6.1]{wise_coherence_2022}. Adapting his argument may be sufficient to prove Conjecture \ref{conj: Irreducible curvature of mapping tori of free group endomorphisms}. 

The question of which of these examples have negative surface curvature is likely to be more subtle. The following conjecture is the most optimistic possible statement.

\begin{conjecture}[Surface curvatures of mapping tori of free group endormorphisms]\label{conj: Surface curvature of mapping tori of free group endomorphisms}
If $X$ is a compact mapping torus of an injective endomorphism of a finitely generated free group then $\sigma_+(X)< 0$ unless $\pi_1(X)$ has a Baumslag--Solitar subgroup.
\end{conjecture}

Caution is needed here! Example \ref{eg: Brady--Crisp} below provides an example of a 2-complex $X$ with $\sigma_+(X)=0$ but $G=\pi_1(X)$ is hyperbolic, so in particular has no Baumslag--Solitar subgroups. Furthermore, $G$ is free-by-cyclic. However, the 2-complex $X$ is not constructed from the representation of $G$ as a free-by-cyclic group. This example emphasises the importance of choosing the 2-complex $X$ correctly.

Graphs of free groups with cyclic edge groups have also been a popular source of examples. I expect the following conjecture to be quite easy to prove, but it will be a useful warm up for more challenging examples.

\begin{conjecture}[Graphs of free groups with cyclic edge groups]\label{conj: Curvatures for graphs of free groups with cyclic edge groups}
If $X$ is a compact, connected, irreducible graph of graphs with circular edge spaces, then $\rho_+(X)\leq 0$. Furthermore, $\sigma_+(X)<0$ unless $\pi_1(X)$ has a Baumslag--Solitar subgroup.
\end{conjecture}

In the case of general graphs of free groups, it is unclear what the correct classification should be. However, finding it may shed further light on this mysterious class.

\begin{problem}[General graphs of free groups]\label{prob: Curvatures of general graphs of free groups}
Determine when a compact, connected graph of graphs $X$ has non-positive or negative surface or irreducible curvatures.
\end{problem}

\subsection{A botanical map}

We finish this section by computing, or at least estimating, the curvatures of several examples. The intention is to exhibit examples with every combination of signs of $\rho_+$ and $\sigma_+$ allowed by Theorem \ref{thm: Curvature inequalities}. The examples are illustrated in Figure \ref{fig: Conjectural botany}. 

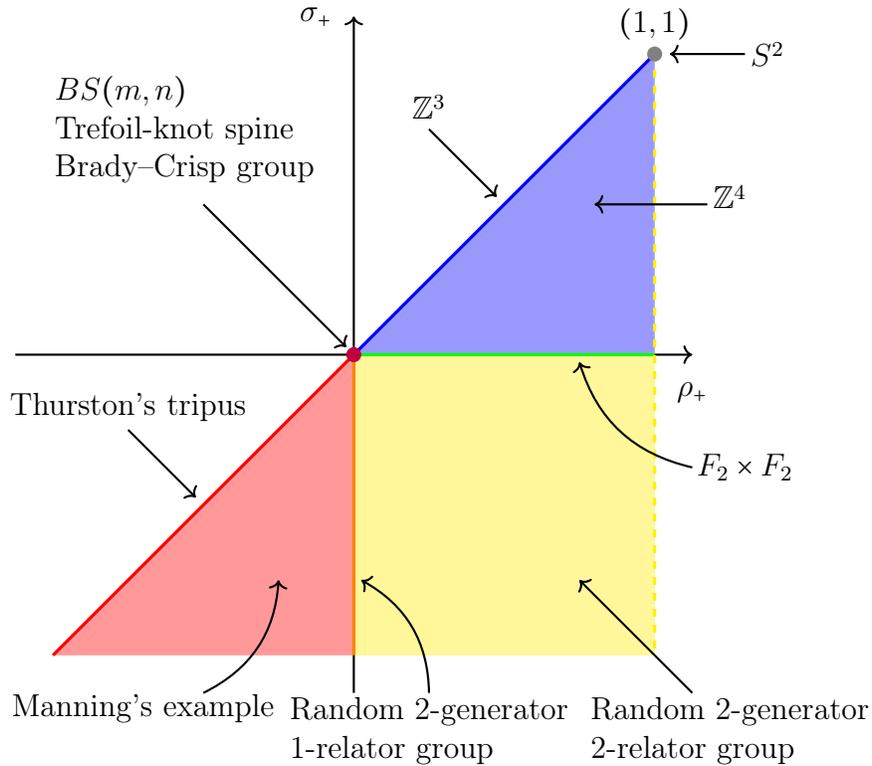
\begin{figure}[htp]
\begin{center}
\begin{tikzpicture}
\fill[blue!40!white] (0,0) -- (4,4) -- (4,0) -- cycle;
\fill[red!40!white] (0,0) -- (-4,-4) -- (0,-4) -- cycle;
\fill[yellow!50!white] (0,0) rectangle (4,-4);
\draw[thick,->] (-4.5,0) -- (4.5,0);
\draw[thick,->] (0,-4.5) -- (0,4.5);
\draw[very thick, blue] (0,0) -- (4,4);
\draw[very thick, red] (0,0) -- (-4,-4);
\draw[very thick, orange] (0,0) -- (0,-4);
\draw[very thick, green] (0,0) -- (4,0);
\draw[very thick, yellow, dashed] (4,4) -- (4,-4);
\fill[purple] (0,0) circle (0.1);
\fill[gray] (4,4) circle (0.1);
\node at (-0.5,4.5) {$\sigma_+$};
\node at (4.5,-0.5) {$\rho_+$};
\draw[thick,->] (-2,2) -- (-0.1,0.1);
\node[align=left] at (-2.25,3) {$BS(m,n)$\\Trefoil-knot spine\\Brady--Crisp group};
\draw[thick, ->] (-2,-4.5) to [bend right] (-1,-3);
\node[align=left] at (-2.8,-4.7) {Manning's example};
\draw[thick, ->] (1,-4.5) to [bend right]  (0.1,-3);
\node[align=left] at (1,-5) {Random 2-generator\\1-relator group};
\draw[thick, ->] (4.5,-4.5) to (3,-3);
\node[align=left] at (5,-5) {Random 2-generator\\2-relator group};
\draw[thick, ->] (4.5,-1.5) to [bend left]  (3,-0.1);
\node[align=left] at (5.2,-1.5) {$F_2\times F_2$};
\node[align=left] at (4,4.4) {$(1,1)$};
\draw[thick, ->] (5.2,4) to (4.2,4);
\node at (5.5,4) {$S^2$};
\draw[thick, ->] (4.7,2) to (3.2,2);
\node at (5,2.1) {$\mathbb{Z}^4$};
\draw[thick, ->] (-3,-1) to (-2.1,-1.9);
\node[align=left] at (-3,-0.7) {Thurston's tripus};
\draw[thick, ->] (1,3) to (1.9,2.1);
\node[align=left] at (1,3.3) {$\mathbb{Z}^3$};
\end{tikzpicture}
\caption{Some examples of standard, finite, irreducible 2-complexes $X$ (sometimes referred to by their fundamental groups) located according to the values $\sigma_+(X)$ and $\rho_+(X)$. See Examples \ref{eg: 2-sphere} for $S^2$, \ref{eg: n-torus} for $\Z^3$ and $\Z^4$, \ref{eg: FxF} for $F_2\times F_2$, \ref{eg: Trefoil knot} for the trefoil-knot complement, \ref{eg: Baumslag--Solitar groups revisited} for $BS(m,n)$, \ref{eg: Brady--Crisp} for the Brady--Crisp group, \ref{eg: Tripus} for Thurston's tripus and \ref{eg: Manning's example} for Manning's example. Random 2-generator, 2-relator groups have $\rho_+(X)>0>\sigma_+(X)$ by Example \ref{eg: Few-relator complexes}, while random 2-generator, 1--relator groups have $\sigma_+(X)<0$ by Example \ref{eg: Few-relator complexes} and $\rho_+(X)=0$ by Corollary \ref{cor: Irreducible curvature of two-generator, one-relator groups}.
}\label{fig: Conjectural botany}
\end{center}
\end{figure}

We first give an example of a 2-complex at the apex.

\begin{example}[2-sphere]\label{eg: 2-sphere}
The standard cellulation $X$ of the 2-sphere $S^2$ with two 2-cells has $\kappa(X)=1$, and so $\sigma_+(X)=\rho_+(X)=1$ by Example \ref{eg: Surfaces}.
\end{example}

More generally, any 2-complex $X$ at the apex fails to be concise: it either has two faces with homotopic attaching maps, or it has a face for which the attaching map is a proper power.

By Proposition \ref{prop: rho_+=1}, the region $\sigma_+(X)<\rho_+(X)=1$ is empty. The 2-skeleta of the $n$-tori provide convenient examples with $\sigma_+(X)>0$. 

\begin{example}[$n$-torus]\label{eg: n-torus}
Let $X$ be the standard presentation complex for $\mathbb{Z}^n$, which can also be thought of as the 2-skeleton of the standard cellulation of the $n$-torus.

If $n=2$ then $\rho_+(X)=\sigma_+(X)=0$ by Example \ref{eg: Surfaces}.

For $n\geq 3$, we compute
\[
\kappa(X)=\frac{1-n+\binom{n}{2}}{\binom{n}{2}} = 1-2/n
\]
so $1>\rho_+(X)\geq 1-2/n$.\footnote{I expect this last estimate to be an equality; it should be possible to check this using Wise's sectional curvature for angle structures \cite{wise_sectional_2004}.}

To compute $\sigma_+(X)$, on the one hand Corollary \ref{cor: Surface curvature and girths of links} gives that $\sigma_+(X)\leq 1/3$ since every face of $X$ is a square and the link of the unique vertex has girth 3. On the other hand, since $n\geq 3$ there is an obvious immersion $C\to X$, where $C$ is the 2-skeleton of a cube, and so $\sigma_+(X)\geq \kappa(C)=1/3$. In summary, we have
\[ 
\sigma_+(X)=1/3 \leq1-2/n\leq \rho_+(X)<1\,.
\]
If $n=3$, $X$ is geometric and so we have $\rho_+(X)=\sigma_+(X)=1/3$, while for $n>3$ we have $\rho_+(X)>\sigma_+(X)$.
\end{example}

Thus, $n$-tori provide a family of examples in the region $0<\sigma_+(X)\leq\rho_+(X)<1$.  Moving downwards in Figure \ref{fig: Conjectural botany}, the direct product of two free groups provides an easy example of a 2-complex with $0=\sigma_+(X)<\rho_+(X)$.

\begin{example}[$F_2\times F_2$]\label{eg: FxF}
The standard complex $X$ associated to the presentation
\[
\langle a_1,a_2,b_1,b_2\mid [a_i,b_j]=1\,,\,i,j=1,2\rangle
\]
of the group $F_2\times F_2$ is a locally $CAT(0)$ Euclidean complex, so has $\sigma_+(X)= 0$. Indeed, Corollary \ref{cor: Surface curvature of CAT(0) 2-complexes} implies that $\sigma_+(X)\leq 0$, while the obvious geodesic inclusion of a torus implies that $\sigma_+(X)\geq 0$ by monotonicity. On the other hand,
\[
\rho_+(X)\geq \frac{\tau(X)}{\Area(X)} = 1/4>0\,.
\]
\end{example}

Likewise, there are many examples with $\sigma_+(X)<0<\rho_+(X)$.  For instance, we saw in Example \ref{eg: Few-relator complexes} that the standard complex associated to a randomly chosen 2-generator, 2-relator presentation has this property.

We now turn our attention to examples with $\rho_+(X)\leq 0$. The origin $\sigma_+(X)=\rho_+(X)=0$ exhibits an especially rich array of behaviours. The 2-torus is the most obvious example, but we list three more illustrative examples here.

\begin{example}[Trefoil knot]\label{eg: Trefoil knot}
Let $X$ be the standard presentation complex of
\[
\langle a,b\mid a^2b^{-3}\rangle\,,
\]
the fundamental group of the trefoil-knot complement. As a one-relator presentation complex it has $\rho_+(X)\leq 0$ by Theorem \ref{thm: Irreducible curvature of one-relator groups}, but $\kappa(X)=0$ so $\rho_+(X)=0$. Since $X$ is geometric, $\sigma_+(X)=\rho_+(X)=0$ by Corollary \ref{cor: Curvatures of 3-manifold spines}.
\end{example}

We already saw that presentation complexes of Baumslag--Solitar groups have $\rho_+(X)=0$. In fact, the same is true of their maximal surface curvature.

\begin{example}[Baumslag--Solitar groups revisited]\label{eg: Baumslag--Solitar groups revisited}
Let $X$ be the presentation complex of
\[
\langle a,b\mid ba^mb^{-1}a^{-n}\rangle\,,
\]
the standard presentation of the $(m,n)$ Baumslag--Solitar group. We saw that $\rho_+(X)=0$ in Example \ref{eg: Baumslag-Solitar groups}. But \cite[Theorem 21]{gordon_surface_2010} asserts that $X$ is virtually geometric, so $\sigma_+(X)=0$ by Corollary \ref{cor: Curvatures of virtually geometric 2-complexes}.
\end{example}

The next example is less well known, but is important because it provides an example with $\rho_+(X)=\sigma_+(X)=0$ and with hyperbolic fundamental group. This example was studied by Brady and Crisp \cite{brady_CAT0_2007}, and we name it for them,

\begin{example}[Brady--Crisp group]\label{eg: Brady--Crisp}
Let $X$ be the presentation complex of the \emph{Brady--Crisp group}
\[
G=\langle a,b\mid aba^2b^{-2}\rangle\,.
\]
As a 2-generator, one-relator presentation, we have $\rho_+(X)=0$ by Corollary \ref{cor: Irreducible curvature of two-generator, one-relator groups}. The link $L$ of the unique vertex of $X$ is a tetrahedron, which is in particular a 3-valent planar graph. But any such 2-complex is virtually geometric, via an argument similar to \cite[\S6]{gordon_surface_2010}. We sketch the argument here. 

Because $L$ is planar, we can think of it as embedded on the surface of a 3-ball $B$. We now thicken each vertex $v$ of $L$ to a disc $D_v$ on $\partial B$. Each vertex of $L$ has an opposite vertex $\opp{v}$. The edges of $L$ incident at $v$ are identified with the edges of $L$ incident at $\opp{v}$, and since $L$ is trivalent we can extend this identification to a homeomorphism $D_v\cong D_{\opp{v}}$.  Gluing up the discs in pairs using these homeomorphisms constructs a (possibly non-orientable) handlebody $U$, which is naturally a thickening of the 1-skeleton $X_{(1)}$, and by construction the attaching curves of the faces are embedded on the boundary $\partial U$. This does not show that $X$ is geometric, because the curves may not be 2-sided. However, passing to an orientable double cover $U'\to U$ and pulling back the attaching maps of the faces constructs a degree-two branched cover $X'\to X$ which is certainly geometric, because every embedded curve on an orientable surface is 2-sided.

In summary, $X$ is virtually geometric, so $\sigma_+(X)=\rho_+(X)=0$ by Corollary \ref{cor: Curvatures of virtually geometric 2-complexes}.
\end{example}

The example of the Brady--Crisp group emphasises that, in some sense, our invariants only see the 2-dimensional features of a 2-complex. Indeed, Brady and Crisp studied $G$ as an example of a hyperbolic group with CAT(0) dimension $2$ but CAT(-1) dimension $3$. If it were possible to put a polyhedral CAT(-1) structure on $X$ then we would have $\sigma_+(X)<0$, by Corollary \ref{cor: Surface curvature of CAT(-1) 2-complexes}.

Any $C(7)$ two-generator, one-relator presentation (for instance, a randomly chosen one will do) provides an example with $\rho_+(X)=0$ but $\sigma_+(X)<0$.

Finally, we give examples with $\rho_+(X)<0$. Corollary \ref{cor: Curvatures of 3-manifold spines} suggests the spine of a hyperbolic 3-manifold with totally geodesic boundary. Thurston's tripus provides a famous example.

\begin{example}[Thurston's tripus]\label{eg: Tripus}
In \cite[\S3.2]{thurston_geometry_1979}, Thurston constructed a compact hyperbolic 3-manifold $T$ with boundary a totally geodesic surface of genus 2, by identifying faces of a hyperbolic polyhedron. The manifold $T$ is sometimes called the \emph{tripus}. Paoluzzi--Zimmerman computed a presentation of the fundamental group as
\[
\pi_1(T)=\langle a,b,c\mid ab^{-1}a^{-1}bc^{-1}b^{-1}ca^{-1}c^{-1}\rangle\,,
\]
and indeed this presentation can be seen as dual to the polyhedron description, which means that the presentation complex $X$ is geometric and $M(X)\cong T$ \cite{paoluzzi_class_1996}.

The presentation complex $X$ provides an example with $\rho_+(X)=\sigma_+(X)<0$. To see this, note that the link of the unique vertex of $X$ is the 1-skeleton of a triangular prism, and in particular has girth equal to 3. The face has length 9 and so
\[
\sigma_+(X)\leq 1-\frac{9}{2}\left( 1- \frac{2}{3}\right) = -\frac{1}{2}
\]
by Corollary \ref{cor: Surface curvature and girths of links}. Meanwhile, $\rho_+(X)=\sigma_+(X)$ since $X$ is geometric.
\end{example}

Finally, we finish with an example satisfying $\sigma_+(X)<\rho_+(X)<0$. This example was studied by Manning \cite{manning_virtually_2010}, who exhibited it as the first example of a one-relator presentation complex that is not virtually geometric.  The invariants $\rho_+$ and $\sigma_+$ enable us to recover Manning's result by other means.

\begin{example}[Manning's example]\label{eg: Manning's example}
Let $X$ be the standard complex associated to the presentation
\[
\langle a,b,c\mid b^2a^2c^2abc\rangle\,.
\]
The link of the vertex is the complete bipartite graph $K_{3,3}$ \cite[Example 2.1]{manning_virtually_2010}, which has girth 4, so
\[
\sigma_+(X)\leq 1-\frac{9}{2}\left(1-\frac{2}{4}\right)=-5/4
\]
by Corollary \ref{cor: Surface curvature and girths of links}. On the other hand, $\rho_+(X)\geq \chi(X)=-1$. Therefore, $X$ provides an example with
\[
\sigma_+(X)<\rho_+(X)\,.
\]
By Corollary \ref{cor: Curvatures of virtually geometric 2-complexes}, it follows that $X$ is not virtually geometric, which gives an alternative proof of \cite[Corollary 4.2]{manning_virtually_2010}

To prove that $\rho_+(X)<0$ requires a more delicate argument.  One approach is to compute the primitivity rank,
\[
\pi(b^2a^2c^2abc)=3\,.
\]
Theorem \ref{thm: Primitivity rank is an absolute extremum} now implies that every $Y\in\Irred(X)$ has $\tau(Y)\leq -1$, so they also have $\kappa(Y)<0$, and therefore $\rho_+(X)<0$ by the rationality theorem. 

Alternatively, for an effective estimate, one can put an angle structure on $X$ and adapt Wise's generalised sectional curvature to the setting of branched immersions, as suggested in Remark \ref{rem: Wise's sectional curvature}. In fact, setting the length of every edge in the link equal to $5\pi/9$ leads to the estimate $\rho_+(X)\leq -1$, and since $\chi(X)=-1$ we actually have equality.
\end{example}

\section{Properties of groups with curvature bounds}\label{sec: Properties of groups with curvature bounds}

In this section, we collect known consequences of curvature bounds on 2-complexes, each of which can be considered as a contribution to the geography problem.  The currently known consequences all first appeared in work on one-relator groups. We also make some further conjectures. The conjectural picture is summarised in a geographical map -- see Figure \ref{fig: Conjectural geography}.

\subsection{Concise 2-complexes}\label{sec: Concise 2-complexes}

We start by noting the proof of Proposition \ref{prop: rho_+=1}, which classifies the 2-complexes with $\rho_+(X)=1$.

\begin{proof}[Proof of Proposition \ref{prop: rho_+=1}]
To see that (\ref{prop: rho_+=1 item (i)}) implies (\ref{prop: rho_+=1 item (iii)}), suppose that $\rho_+(X)=1$. By the rationality theorem, $\rho_+(X)$ is realised by an essential map from an irreducible 2-complex $Y$ with $\kappa(Y)=1$, so $\chi(Y_{(1)})=0$. Without loss of generality $Y$ is connected, so $Y_{(1)}$ is a circle; write $\pi_1(Y_{(1)})=\langle w\rangle$. Therefore, the attaching maps of the faces of $Y$ represent non-zero powers of $w$ in the fundamental group. If $Y$ were concise then there could be only one such face, and it would be represented by $w$, but this would contradict the fact that $Y$ is irreducible. So $X$ is not concise.

To show that (\ref{prop: rho_+=1 item (iii)}) implies (\ref{prop: rho_+=1 item (ii)}), note that, if $X$ is not concise, then there is an immersion $Y\to X$, where as above $Y$ is an irreducible 2-complex with 1-skeleton a circle. Now $Y$ is an irrigid complex in the sense of Defiinition \ref{def: Surface decomposition and irrigid complexes} and so, by Proposition \ref{lem: Surface to surface type}, there is an essential map from a surface $S\to Y$ with $\kappa(S)=1$. Thus
\[
1=\kappa(S)\leq \sigma_+(Y)\leq \sigma_+(X)
\]
by Lemma \ref{lem: Monotonicity}, so $\sigma_+(X)=1$ by Theorem \ref{thm: Curvature inequalities}.

Finally, (\ref{prop: rho_+=1 item (ii)}) implies (\ref{prop: rho_+=1 item (i)}) immediately from Theorem \ref{thm: Curvature inequalities}.
\end{proof}

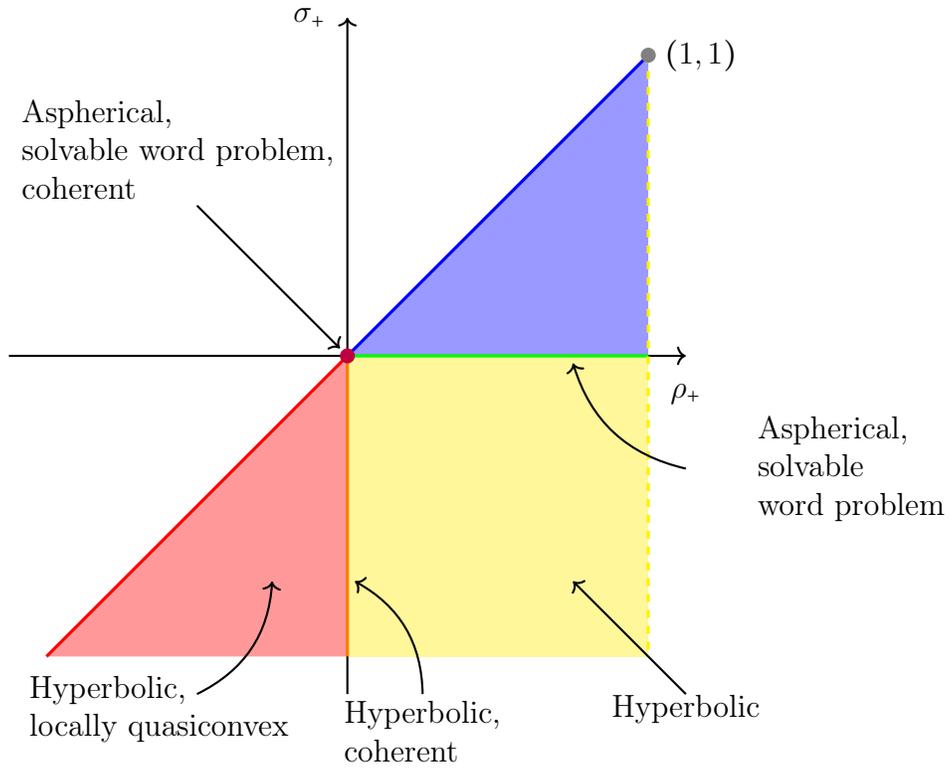
\begin{figure}[htp]
\begin{center}
\begin{tikzpicture}
\fill[blue!40!white] (0,0) -- (4,4) -- (4,0) -- cycle;
\fill[red!40!white] (0,0) -- (-4,-4) -- (0,-4) -- cycle;
\fill[yellow!50!white] (0,0) rectangle (4,-4);
\draw[thick,->] (-4.5,0) -- (4.5,0);
\draw[thick,->] (0,-4.5) -- (0,4.5);
\draw[very thick, blue] (0,0) -- (4,4);
\draw[very thick, red] (0,0) -- (-4,-4);
\draw[very thick, orange] (0,0) -- (0,-4);
\draw[very thick, green] (0,0) -- (4,0);
\draw[very thick, yellow, dashed] (4,4) -- (4,-4);
\fill[purple] (0,0) circle (0.1);
\fill[gray] (4,4) circle (0.1);
\node at (-0.5,4.5) {$\sigma_+$};
\node at (4.5,-0.5) {$\rho_+$};
\draw[thick,->] (-2,2) -- (-0.1,0.1);
\node[align=left] at (-2.25,2.75) {Aspherical,\\solvable word problem,\\coherent};
\draw[thick, ->] (-2,-4.5) to [bend right] (-1,-3);
\node[align=left] at (-2.5,-4.7) {Hyperbolic,\\locally quasiconvex};
\draw[thick, ->] (1,-4.5) to [bend right]  (0.1,-3);
\node[align=left] at (1,-5) {Hyperbolic,\\coherent};
\draw[thick, ->] (4.5,-4.5) to (3,-3);
\node[align=left] at (4.5,-4.7) {Hyperbolic};
\draw[thick, ->] (4.5,-1.5) to [bend left]  (3,-0.1);
\node[align=left] at (6.7,-1.5) {Aspherical,\\solvable\\word problem};
\node[align=left] at (4.7,4) {$(1,1)$};
\end{tikzpicture}
 \caption{Conjectural properties of a finite, irreducible 2-complex $X$ and its fundamental group $G=\pi_1(X)$, depending on the invariants $\sigma_+(X)$ and $\rho_+(X)$. The vertex $\sigma_+(X)=\rho_+(X)=1$ consists of 2-complexes with either two equal faces or a single face for which the attaching map is a proper power. The dashed line $\sigma_+(X)<\rho_+(X)=1$ is empty.}\label{fig: Conjectural geography}
\end{center}
\end{figure}

\subsection{Non-positive irreducible curvature}

Recall that a map of 2-complexes $Y\to X$ is an \emph{immersion} if it is locally injective everywhere. In 2003, Wise introduced the following definition \cite{wise_nonpositive_2003}.

\begin{definition}[Non-positive immersions]\label{def: Non-positive immersions}
A finite, standard 2-complex $X$ has \emph{non-positive immersions} if, for any immersion from a finite, connected, standard 2-complex $Y\to X$, either $\chi(Y)\leq 0$ or $Y$ is simply connected. If $\chi(Y)>0$ implies that $Y$ is contractible, then $X$ is said to have \emph{contracting non-positive immersions}.
\end{definition}
The ideas around this definition have been developed by Wise more recently in \cite{wise_coherence_2022} and \cite{wise_invitation_2020}.

It is not hard to see that the non-positive-immersions property is a slight weakening of our notion of non-positive irreducible curvature.

\begin{lemma}[Non-positive irreducible curvature implies non-positive immersions]\label{lem: Non-positive irreducible curvature implies non-positive immersions}
Let $X$ be a finite, standard 2-complex. If $\rho_+(X)\leq 0$ then $X$ has contracting non-positive immersions.
\end{lemma}
\begin{proof}
For any immersion from a finite, connected complex  $Y\to X$, consider an unfolding
\[
Y'=G\vee Y_1\vee\ldots\vee Y_m\vee D_1\vee\ldots\vee D_n
\]
as in Lemma \ref{lem: Unfolding 2-complexes}.  Since $Y$ and $Y'$ are homotopy equivalent, the Euler characteristic of $Y$ can be written as
\[
\chi(Y)=\chi(G)+\sum_{i=1}^m(\chi(Y_i)-1)\,.
\]
The graph $G$ is connected so $\chi(G)\leq 1$, and the hypothesis on $\rho_+(Y_i)$ implies that
\[
\chi(Y_i)=\tau(Y_i)\leq 0\,,
\]
so $\chi(Y)\leq 1$ with equality only if $G$ is a tree and $m=0$. Therefore, $\chi(Y)>0$ implies that $Y$ is homotopy equivalent to a tree wedged with discs, and is in particular contractible, as required.
\end{proof}

It is an open question whether there is an example of a 2-complex $X$ with contracting non-positive immersions but $\rho_+(X)>0$. In any case, I believe that the class of 2-complexes with non-positive irreducible curvature may be more useful than the possibly larger class of 2-complexes with non-positive immersions. Indeed, the fact that non-positive irreducible curvature can be recognised algorithmically is one argument in favour of this point of view.

In any case, Lemma \ref{lem: Non-positive irreducible curvature implies non-positive immersions} enables us to exploit the known properties of 2-complexes with non-positive immersions that Wise has established.  The following theorem, which contributes to the geography problem, follows from \cite[Theorems 1.3, 1.6]{wise_coherence_2022} using Lemma \ref{lem: Non-positive irreducible curvature implies non-positive immersions}. Recall that a group $G$ is said to be \emph{locally indicable} if every non-trivial, finitely generated subgroup $H$ of $G$ surjects $\Z$.

\begin{theorem}[Geography with non-positive irreducible curvature]\label{thm: Geography with non-positive irreducible curvature}
Let $X$ be a finite, standard, irreducible 2-complex. If $\rho_+(X)\leq 0$ then:
\begin{enumerate}[(i)]
\item $X$ is aspherical, and
\item $\pi_1(X)$ is locally indicable.
\end{enumerate}
\end{theorem}

The proofs in \cite{wise_coherence_2022} use tower arguments, a technique going back to Papakyriakopoulos' proof of Dehn's lemma \cite{papakyriakopoulos_dehn_1957} and also used by Howie in his locally indicable Freiheitssatz  \cite{howie_pairs_1981}.  The folding technique of Lemma \ref{lem: Folding 2-complexes} can also be used to give a direct proof. Combined with Theorem \ref{thm: Irreducible curvature of one-relator groups}, the theorem recovers the Cockroft--Lyndon asphericity theorem for torsion-free one-relator groups \cite{lyndon_cohomology_1950,cockcroft_two-dimensional_1954}, as well as Brodski\u{\i}--Howie's theorem that torsion-free one-relator groups are locally indicable \cite{brodski_equations_1980,howie_locally_1982}.

A group $G$ is \emph{coherent} if every finitely generated subgroup is finitely presented.  A notorious question of Baumslag, recently answered by Jaikin-Zapirain--Linton \cite{jaikin-zapirain_coherence_2023}, asks whether every one-relator group is coherent.  Motivated by Baumslag's question, Wise conjectured that the fundamental group of any finite 2-complex $X$ with non-positive immersions is coherent \cite[Conjecture 1.10]{wise_coherence_2022}.  The following conjecture sharpens Wise's conjecture slightly.

\begin{conjecture}[Non-positive irreducible curvature implies coherence]\label{conj: Non-positive irreducible curvature and coherence}
Let $X$ be a finite, standard, irreducible 2-complex. If $\rho_+(X)\leq 0$ then $\pi_1(X)$ is coherent.
\end{conjecture}

We shall see below that this conjecture is true when $\rho_+(X)<0$.

\subsection{Negative irreducible curvature}

Negative irreducible curvature puts even stronger constraints on a 2-complex $X$. Recently, Louder and the author strengthened Wise's definition of non-positive immersions to \emph{negative immersions} \cite[Definition 1.1]{louder_negative_2022}, and then further to \emph{uniform negative immersions} \cite[Definition 3.23]{louder_uniform_2024}.

\begin{definition}\label{def: UNI}
A finite, standard 2-complex $X$ has \emph{negative immersions} if, for every immersion from a finite, connected complex $Y\to X$, either $\chi(Y)<0$ or $Y$ is reducible. If, furthermore,  there is an $\epsilon>0$ such that
\[
\kappa(Y)=\frac{\chi(Y)}{\Area(Y)}\leq -\epsilon
\]
for every such immersion $Y\to X$, then $X$ has \emph{uniform negative immersions}.
\end{definition}

\begin{remark}\label{rem: Wise's NI}
Wise makes use of a different definition of negative immersions for 2-complexes in \cite{wise_coherence_2022}. See \cite[\S3.4]{louder_uniform_2024} for a comparison between Definition \ref{def: UNI} and Wise's definition.
\end{remark}

Negative irreducible curvature is clearly stronger than uniform negative immersions, since the definitions are the same except that the class of maps considered expands from immersions to branched immersions. The next lemma is therefore immediate.

\begin{lemma}[Negative irreducible curvature implies uniform negative immersions]\label{lem: Negative irreducible curvature implies UNI}
If $X$ is a finite, standard, 2-complex and $\rho_+(X)<0$ then $X$ has uniform negative immersions.
\end{lemma}

The following theorem, which contributes to the geography problem for complexes of negative irreducible curvature, lists some very strong constraints on the finitely generated subgroups of the fundamental groups of complexes with negative irreducible curvature.

\begin{theorem}[Geography with negative irreducible curvature]\label{thm: Subgroups of groups with negative irreducible curvature}
If $X$ is a finite, irreducible, standard 2-complex and $\rho_+(X)<0$, then the fundamental group $G$ satisfies the following conditions:
\begin{enumerate}[(i)]
\item $G$ is 2-free (i.e.\ every 2-generator subgroup is free); 
\item $G$ is coherent (i.e.\ every finitely generated subgroup is finitely presented);
\item each finitely generated, one-ended subgroup $H$ of $G$ is co-Hopfian (i.e.\ every injective self-homomorphism of $H$ is an automorphism);
\item for each integer $r$, there are only finitely many conjugacy classes of finitely generated, one-ended subgroups $H$ of $G$ such that the abelianisation of $H$ has rational rank at most $r$;
\item each finitely generated subgroup $H$ of $G$ is either cyclic or large in the sense of Pride (i.e.\ some subgroup $H_0$ of finite index in $H$ surjects a non-abelian free group); 
\end{enumerate}
\end{theorem}
\begin{proof}
Items (ii), (iii), (iv) and (v) are the immediate results of combining Lemma \ref{lem: Negative irreducible curvature implies UNI} with \cite[Corollary 6.4 and Theorem 6.5]{louder_uniform_2024}.

Item (i) is implicit in \cite{louder_negative_2022}, but we give an explicit proof here. Suppose that $H$ is a 2-generator subgroup of $G$, which we may assume is freely indecomposable by Grushko's theorem. Now $H$ is represented up to conjugacy by an immersion $Y\to X$, where $Y$ is a finite, connected, irreducible 2-complex, by \cite[Theorem 6.3]{louder_uniform_2024}.  The hypothesis that $\rho_+(X)<0$ implies that
\[
\tau(Y)=\chi(Y)=1-b_1(Y)+b_2(Y)<0
\]
so, since $b_1(Y)\leq 2$, it follows that both $b_1(Y)=2$ and $b_2(Y)=0$. Finally, a theorem of Stallings implies that $H$ is free, as required \cite{stallings_homology_1965}.
\end{proof}

The combination of Theorems \ref{thm: NI for one-relator groups} and \ref{thm: Subgroups of groups with negative irreducible curvature} recover the properties of one-relator groups with negative immersions discovered in \cite{louder_negative_2022} and \cite{louder_uniform_2024}.

Theorems \ref{thm: Geography with non-positive irreducible curvature} and \ref{thm: Subgroups of groups with negative irreducible curvature}(i) together imply that if $\rho_+(X)<0$ then $G=\pi_1(X)$ is a group of finite type without Baumslag--Solitar subgroups, which naturally raises the question of whether $G$ is hyperbolic. If so, then one might further hope that the some of the constraints on the subgroups of $G$ implied by Theorem \ref{thm: Subgroups of groups with negative irreducible curvature} come from the geometry of those subgroups. These considerations motivate the following conjecture, which sharpens \cite[Conjecture 1.9]{louder_negative_2022} and \cite[Conjecture 5.1]{wise_coherence_2022}.

\begin{conjecture}[Negative irreducible curvature implies locally quasiconvex]\label{conj: Negative irreducible curvature implies hyperbolic and locally quasiconvex}
Let $X$ be a finite, standard, irreducible 2-complex. If $\rho_+(X)< 0$ then $\pi_1(X)$ is a locally quasiconvex hyperbolic group.
\end{conjecture}

This should be compared to Conjecture \ref{conj: Negative surface curvature implies hyperbolic} below, which relaxes the hypotheses to $\sigma_+(X)<0$. Italiano, Martelli and Migliorini recently gave the first example of a non-hyperbolic group of finite type without Baumslag--Solitar subgroups \cite{italiano_hyperbolic_2023}, but no 2-dimensional examples are known. A recent theorem of Linton  \cite{linton_one-relator_2022} proves that one-relator groups with negative immersions -- i.e.\ with $\rho_+(X)<0$ -- are hyperbolic, confirming part of Conjecture \ref{conj: Negative irreducible curvature implies hyperbolic and locally quasiconvex} in the one-relator case.

\subsection{Surface curvature bounds}

As we have seen in the previous two sections, recent developments in the theory of one-relator groups can be seen as explorations of the consequences of upper bounds on irreducible curvature. Surface curvature has not yet been as thoroughly investigated, and so the constraints that surface curvature bounds impose on $X$ are still conjectural. We list some of those conjectures here. We have already seen many examples of complexes with upper bounds on their surface curvature, and these conjectures are consistent with those examples. Throughout, $X$ is a finite, standard, irreducible 2-complex.

Non-positive curvature is frequently used in geometric group theory to prove that a space is aspherical: this is the Cartan--Hadamard theorem for CAT(0) spaces \cite[Theorem II.4.1]{bridson_metric_1999}, a standard result about small-cancellation complexes \cite{chiswell_aspherical_1981}, and the Lyndon asphericity theorem for one-relator groups \cite{lyndon_cohomology_1950,cockcroft_two-dimensional_1954}. We have already seen that $X$ is aspherical if it has non-positive irreducible curvature, but we conjecture that this result extends to the much wider context of non-positive surface curvature.

\begin{conjecture}[Non-positive surface curvature implies aspherical]\label{conj: Non-positive surface curvature implies aspherical}
If $\sigma_+(X)\leq 0$ then $X$ is aspherical.
\end{conjecture}

A suitable formulation of van Kampen's lemma implies that any map from a 2-sphere to $X$ is homotopic to a morphism of 2-complexes $S\to X$, where the realisation of $S$ is a 2-sphere. The challenge of Conjecture \ref{conj: Non-positive surface curvature implies aspherical} is to improve this morphism to an essential map.

Another frequent use of non-positive curvature in group theory is to the solve the word problem: this follows from the quadratic isoperimetric inequality for CAT(0) groups \cite[Proposition III.$\Gamma$.1.6]{bridson_metric_1999}, from Greendlinger's lemma for small-cancellation groups \cite{lyndon_combinatorial_1977}, and is a famous theorem of Magnus for one-relator groups \cite{magnus_identitatsproblem_1932}.  Our second conjecture proposes that, in the case of 2-complexes, all of these can be deduced from a bound on surface curvature.

\begin{conjecture}[Non-positive surface curvature implies solvable word problem]\label{conj: Non-positive surface curvature implies solvable word problem}
If $\sigma_+(X)\leq 0$ then the word problem is solvable in $\pi_1(X)$.
\end{conjecture}

The last two conjectures suggest that non-positive surface curvature is similar to other notions of non-positive curvature commonly used in geometric group theory. However, as noted after Theorem \ref{thm: Irreducible curvature of one-relator groups}, there are 2-complexes of non-positive surface curvature with very large isoperimetric inequality, and in particular $\sigma_+(X)$ does not imply any of the standard notions of non-positive curvature already used in geometric group theory.

However, it is likely that negative surface curvature is closely related to the dominant notion of negative curvature for groups, namely word-hyperbolicity.

\begin{conjecture}[Negative surface curvature implies hyperbolic]\label{conj: Negative surface curvature implies hyperbolic}
If $\sigma_+(X)<0$ then $\pi_1(X)$ is word-hyperbolic.
\end{conjecture}

This conjecture posits a common generalisation of Greendlinger's lemma for small-cancellation groups and the B.B.\ Newman spelling theorem for one-relator groups with torsion. It should also shed light on the cases of non-positively curved Euclidean 2-complexes and one-relator presentation complexes, where there is not yet a complete understanding of which groups are word-hyperbolic. However, Example \ref{eg: Brady--Crisp} provides an example of a non-positively curved, Euclidean, one-relator 2-complex, for which hyperbolicity of the fundamental group is not a consequence of Conjecture \ref{conj: Negative surface curvature implies hyperbolic}.

The conjectures to this point all concern upper curvature bounds. However, it is possible that lower bounds may also provide useful information.

\begin{conjecture}[Positive surface curvature implies no surface subgroups]\label{conj: Surface subgroups and positive curvature}
If $\sigma_-(X)>0$ then no subgroup of $\pi_1(X)$ is isomorphic to the fundamental group of a closed surface $S$ with $\chi(S)\leq 0$.
\end{conjecture}

The challenge of Conjecture \ref{conj: Surface subgroups and positive curvature} is to prove that \emph{any} finite 2-complex $X$ with $\pi_1(X)$ a surface group has $\rho_-(X)\leq 0$.

In \S\ref{sec: Irrigid complexes}, we saw that Baumslag--Solitar groups are closely related to surfaces, because both can be thought of as the fundamental groups of irrigid complexes. An even more intriguing possibility is that $\sigma_-(X)>0$ also implies that $\pi_1(X)$ has no Baumslag--Solitar subgroups. There are very few known examples of infinite, finitely presented groups without surface or Baumslag--Solitar subgroups.

\bibliographystyle{plain}

\Addresses

\end{document}